\documentclass{article}
\usepackage{graphicx} 
\usepackage{amsthm}
\usepackage{amsfonts, amsmath, amssymb, amsthm}
\usepackage{hyperref}
\usepackage{cleveref}
\usepackage[utf8x]{inputenc}
\usepackage{url}
\usepackage{tikz}
\usepackage{enumerate}
\usepackage{stmaryrd} 

\newcommand{\Psf}{\mathsf{P}}

\def\F{\mathcal{F}}
\def\E{\mathcal{E}}

\def\L{\mathcal{L}}

\def\M{\mathcal{M}}

\newcommand{\bstr}{2^{<\NN}}

\newcommand{\ISig}{\mathsf{I}\Sigma^0}

\newcommand{\BSig}{\mathsf{B}\Sigma^0}

\newcommand{\RCA}[0]{\mathsf{RCA}}
\newcommand{\WKL}[0]{\mathsf{WKL}}

\newcommand{\ADS}[0]{\mathsf{ADS}}
\newcommand{\EM}[0]{\mathsf{EM}}

\newcommand{\RT}[0]{\mathsf{RT}}

\newcommand{\NN}[0]{\mathbb{N}}

\newcommand{\card}{\operatorname{card}}

\newcommand{\finsub}{\subseteq_{\mathtt{fin}}}

\def\qt#1{``#1''}%

\newcommand{\Bsig}{\mathsf{B}\Sigma}

\newcommand{\PAm}{\mathsf{PA}^{-}}

\newcommand{\Cod}{\operatorname{Cod}}
\newcommand{\Jcut}{\mathsf{J}}
\newcommand{\shortbrace}[1]{\stackrel{\scriptsize #1 \mbox{ times}}{\dots}}
\newcommand{\Leaves}{\mathrm{Leaves}}
\newcommand{\Str}{\mathrm{Str}}
\newcommand{\Mil}{\mathrm{Mil}}

\title{Largeness notions and polytime translation\\ for $\forall \Sigma^0_3$-consequences of $\mathsf{RT}^2_2$}
\date{\today}

\newtheorem*{statement}{Statement}
\newtheorem{theorem}{Theorem}
\numberwithin{theorem}{section}
\newtheorem{maintheorem}[theorem]{Main Theorem}
\newtheorem{lemma}[theorem]{Lemma}

\newtheorem{proposition}[theorem]{Proposition}
\newtheorem{remark}[theorem]{Remark}
\newtheorem{definition}[theorem]{Definition}
\newtheorem{corollary}[theorem]{Corollary}

\makeatletter
\newtheorem*{rep@theorem}{\rep@title}
\newcommand{\newreptheorem}[2]{%
\newenvironment{rep#1}[1]{%
 \def\rep@title{#2 \ref{##1}}%
 \begin{rep@theorem}}%
 {\end{rep@theorem}}}
\makeatother

\newreptheorem{theorem}{Theorem}
\newreptheorem{maintheorem}{Main Theorem}

\usepackage{xcolor}

\usepackage{authblk}
\DeclareSymbolFont{bbold}{U}{bbold}{m}{n}
\DeclareMathSymbol{\bbomega}{\mathord}{bbold}{"7F}

\author{Quentin Le Houérou \and Ludovic Patey}

\begin{document}

\maketitle

\begin{abstract}
Le Houérou, Patey and Yokoyama defined a parameterized version of $\alpha$-largeness to prove that $\WKL_0 + \RT^2_2$ is a $\forall \Sigma^0_3$-conservative extension of $\RCA_0 + \BSig_2$, where $\forall \Sigma^0_3$ is the universal set-closure of the class of $\Sigma^0_3$-formulas. We introduce a variant of this notion of largeness and obtain polynomial bounds, using a tree partition theorem based on Milliken's tree theorem. Thanks to the framework of forcing interpretation, this yields that any proof of a $\forall \Sigma^0_3$-sentence in the theory $\WKL_0 + \RT^2_2$ can be translated into a proof in $\RCA_0 + \BSig_2$ at the cost of a polynomial increase in size. 
\end{abstract}

\section{Introduction}

If a computably axiomatized theory~$T_0$ is a conservative extension of another such theory~$T_1$ for a syntactical class of sentences $\Gamma$, there is an algorithmic procedure which translates any $T_0$-proof~$p$ of a $\Gamma$-sentence into its shortest $T_1$-proof~$\hat p$. It is natural to wonder whether the translation $p \mapsto \hat p$ yields significantly longer proofs. If there exists a $T_0$-proof~$p$ of a $\Gamma$-sentence such that the length $|\hat p|$ of its translated $T_1$-proof is super-polynomial with respect to the length~$|p|$, we say that $T_0$ admits \emph{non-trivial speedup} over $T_1$ for $\Gamma$-formulas. In this case, the theory~$T_0$ is arguably useful for $\Gamma$-sentences, in that it sometimes produces significantly shorter proofs than~$T_1$ for such sentences. If on the other hand, there exists a polynomial~$Q$ such that $|\hat p| \leq Q(|p|)$ for every $T_0$-proof~$p$ of any $\Gamma$-sentence, then we consider that $T_0$ admits no significant speedup over~$T_1$ for $\Gamma$-formulas, since the various proof systems are mutually polynomially simulated.

The study of proof size and proof speedups traces back to G\"odel~\cite{davis1990kurt} (see Pudl\'ak~\cite{pudlak1998length} for an excellent survey). However, it was only in the 1990's that Avigad~\cite{avigad1996formalizing} showed that model-theoretic conservation theorems based on forcing could be formalized into proof-theoretic arguments and yield proof size analysis. He proved in particular that $\WKL_0$ is $\Pi^1_1$-conservative over $\RCA_0$ with no significant increase in the length of proofs. The technique was later used to formalize $\Pi^1_1$-conservation theorems over~$\RCA_0+\ISig_n$~\cite{Ikari-PhD,Ik-Yo} and $\forall \Sigma^0_2$-conservation theorems over~$\RCA_0$~\cite{kolodziejczyk2023ramsey} and over $\RCA_0^*$~\cite{katarzyna2025speedup}, where a $\forall \Sigma^0_n$-formula is of the form $\forall X \varphi(X)$ where $\varphi$ is $\Sigma^0_n$.

We are particularly interested in the first-order consequences of Ramsey's theorem for pairs and two colors ($\RT^2_2$). Patey and Yokoyama~\cite{patey2018proof} proved that $\WKL_0 + \RT^2_2$ is $\forall \Sigma^0_2$-conservative over~$\RCA_0$ using the notion of $\alpha$-largeness from Ketonen and Solovay~\cite{ketonen1981rapidly}. The proof was later simplified by Ko{\l}odziejczyk and Yokoyama~\cite{kolo2020some} to give explicit polynomial bounds for this notion of largeness. Ko{\l}odziejczyk, Wong and Yokoyama~\cite{kolodziejczyk2023ramsey} then formalized the construction to prove that the conservation proof does not increase significantly the length of proofs. This constrasts with the fact that $\WKL_0^* + \RT^2_2$ yields a non-elementary speedup over~$\RCA_0^*$ even for $\Sigma_1$-sentences~\cite{kolodziejczyk2023ramsey}. More recently, Le Houérou, Levy Patey and Yokoyama~\cite{houerou2026pi} defined a parameterized version of Ketonen and Solovay's notion of largeness and proved that $\WKL_0 + \RT^2_2$ is $\forall \Sigma^0_3$-conservative over $\RCA_0 + \BSig_2$, with explicit bounds computation. However, these bounds are exponential, leaving open whether the conservation theorem yields a significant proof speedup.

In this article, we define a variant of the parameterized version of $\alpha$-largeness \cite{houerou2026pi} and prove that Ramsey's theorem for pairs admits polynomial bounds with respect to this notion of largeness. Then, using the now-called notion of forcing interpretation introduced by Avigad~\cite{avigad1996formalizing} and developed in the context of $\alpha$-largeness by Ko{\l}odziejczyk, Wong and Yokoyama~\cite{kolodziejczyk2023ramsey}, we prove that the conservation theorem for $\WKL_0 + \RT^2_2$ over $\RCA_0 + \BSig_2$ for $\forall \Sigma^0_3$-sentences does not yield any significant proof speedup. The polynomial bounds computation for this new notion of largeness is non-trivial, and require to prove a partition theorem for trees based on a finite version of Milliken's tree theorem~\cite{milliken1979ramsey,dodos2016ramsey} with primitive recursive bounds. Our main theorem is therefore the following:

\begin{maintheorem}\label[maintheorem]{thm:rt22-polynomially-simulated}
$\RCA_0 + \BSig_2$ polynomially simulates $\WKL_0 + \RT_2^2$ with respect to $\forall \Sigma_3^0$ sentences.
\end{maintheorem}

The article is divided into three parts : in \Cref{sec:largeness}, we survey and compare three related notions of largeness, namely, Ketonen and Solvay's $\alpha$-largeness, Le Houérou, Levy Patey and Yokoyama's $\alpha$-largeness$(\theta)$, and our new notion of $\alpha$-largeness${}^*(\theta)$, for $\alpha$ of the form $\bbomega^n \cdot k$. In \Cref{sec:combinatorics}, we develop the framework of $\alpha$-largeness${}^*(\theta)$ and prove polynomial upper bounds for $\RT^2_2$-$\bbomega^n$-largeness${}^*(\theta)$. Last, in \Cref{sec:translation}, we combine this upper bound with the framework of forcing interpretation to prove \Cref{thm:rt22-polynomially-simulated}.

\subsection{Notation}

Most of the proofs are over a weak subsystem of second-order arithmetic. We therefore distinguish the formal set of integers~$\NN$ from the theory from the standard set of integers~$\omega$ from the meta-theory. For instance, if $\M = (M, S)$ is a non-standard model of second-order arithmetic, $\NN^\M = M$ and $\omega^\M = \{0^\M, 1^\M, \dots \} \subseteq M$. Similarly, we distinguish the formal ordinals of the theory and the ordinals from the meta-theory. We therefore write for instance $\bbomega^\bbomega$ for the formal counter-part of~$\omega^\omega$. Indeed, since $\bbomega^{\bbomega} = \sup_{n \in \NN} \bbomega^n$, if $\omega \subsetneq M$, then $\bbomega$ and a fortiori $\bbomega^\bbomega$ are not well-founded from the viewpoint of the meta-theory.

We often identify an integer $k$ with its ordinal $\{0, \dots, k-1\}$. 
Given a set $X \subseteq \NN$ and an integer~$n \in \NN$, we write $[X]^n$ for the collection of all subsets of~$X$ of size~$n$.
The set $[X]^n$ is in one-to-one correspondence with the set of all increasing ordered $n$-tuples over~$X$. Therefore, for simplicity of notation, given a coloring $f : [X]^n \to k$, we write $f(x_0, \dots, x_{n-1})$ for $f(\{x_0, \dots, x_{n-1}\})$ assuming that $x_0 < \dots < x_{n-1}$.

\section{Notions of largeness}\label[section]{sec:largeness}

Partial conservation theorems are closely related to ordinal computation for appropriate notions of largeness.
In this section, we survey three notions of largeness related to our polytime translation.

\subsection{Largeness and sparsity}

\begin{definition}
A family $\L \subseteq [\NN]^{<\NN}$ is a \textit{largeness notion} if any infinite set has a finite subset in $\L$
and $\L$ is closed under supersets. A finite set~$F$ is \emph{$\L$-large} if $F \in \L$.
\end{definition}

We are particularly interested in the closure of notions of largeness under combinatorial statements coming from Ramsey theory.
Given a function $f : [\NN]^n \to k$ and a finite set~$G = \{x_0 < \dots < x_{\ell-1}\} $, let $f_G : [\ell]^n \to k$ be defined by $f_G(i_0, \dots, i_{n-1}) = f(x_{i_0}, \dots, x_{i_{n-1}})$.

\begin{definition}\label[definition]{def:ramsey-like-Pi12-formula}
Given~$n, k \in \NN$, an \emph{$\RT$-like statement} is a $\Pi_2^1$-formula of the form 
$$(\forall f : [\mathbb{N}]^n \to k)[\Phi(f, \NN) \to (\exists Y)(Y \mbox{is infinite} \wedge \Psi(f,Y))]$$
where $\Phi(f, Y)$ and $\Psi(f, Y)$ are of the form $(\forall G \finsub Y) \Theta(f_G)$ with $\Theta$ a $\Delta^0_0$-formula.
\end{definition}

Ramsey's theorem for $n$-tuples and $k$-colors is the most obvious example of an $\RT$-like statement:
Given a coloring $f : [\NN]^n \to k$, a set $H \subseteq \NN$ is \emph{$f$-homogeneous} if $f$ is constant on $[H]^n$.

\begin{statement}[Ramsey's theorem]
Given $n, k \in \NN$, $\RT^n_k$ is the statement \qt{For every coloring $f : [\NN]^n \to k$, there is an infinite $f$-homogeneous set}.
\end{statement}

In this case, $\Phi_{\RT^n_k}(f,Y)$ is always true, and $\Psi_{\RT^n_k}(f, Y)$ means \qt{$Y$ is $f$-homogeneous}. Ramsey's theorem for pairs and its consequences is a central subject of study in reverse mathematics~\cite{seetapun1995strength,cholak2001strength,liu2012rt22}. In particular, the understanding of its first-order consequences is one of the main open questions~\cite{kolo2021search,patey2018proof,houerou2026pi}.

Our main focus in this article will be to study of largeness closure under applications of $\RT^2_2$. It is however often useful to decompose combinatorial statements into simpler ones. We shall therefore also be interested in the following $\RT$-like statements: Given a coloring $f : [\NN]^2 \to 2$, a set $H \subseteq \NN$ is \emph{$f$-transitive} if for every~$i < 2$ and $x < y < z \in H$, if $f(x, y) = f(y, z) = i$, then $f(x, z) = i$. If $\NN$ is $f$-transitive, we simply say that the coloring $f$ is transitive.

\begin{statement}[Erd\H{o}s-Moser theorem]
$\EM$ is the statement \qt{For every coloring $f : [\NN]^2 \to 2$, there is an infinite $f$-transitive set}.
\end{statement}

\begin{statement}[Ascending Descending Sequence]
$\ADS$ is the statement \qt{For every transitive coloring $f : [\NN]^2 \to 2$, there is an infinite $f$-homogeneous set}.
\end{statement}

The Erd\H{o}s-Moser theorem is originally a statement from graph theory~\cite{erdos1964representation} formulated in terms of tournaments, that is, complete oriented graphs. Infinite graphs are one-to-one correspondence with colorings $f : [\NN]^2 \to 2$. Similarly, the Ascending Descending Sequence is a statement from order theory, stating that every infinite linear order admits an infinite ascending or descending sub-sequence. Here again, linear orders are in one-to-one correspondence with transitive 2-colorings of pairs.

Both statements have been extensively studied in computability theory and reverse mathematics~\cite{bovykin2017strength,lerman2013separating,lerman1982recursive,hirschfeldt2007combinatorial}.
In particular, the first-order part of~$\ADS$ is well-understood, as $\RCA_0 + \ADS$ is a $\Pi^1_1$-conservative extension of~$\RCA_0 + \BSig_2$~\cite{chong2021pi11} and by applying the techniques of \cite{Ikari-PhD}, one can prove that $\RCA_0 + \ADS$ is polynomially simulated by $\RCA_0 + \BSig_2$.

\begin{definition}
Let $\Psf$ be an $\RT$-like statement with witnesses~$n$,$k$, $\Phi$ and~$\Psi$, and let $\L$ be a largeness notion. 
A finite set $F$ is \emph{$\Psf$-$\L$-large} if for every coloring $f : [F]^n \to k$ such that $\Phi(f,\NN)$ holds, there is $\L$-large subset $E \subseteq F$ such that $\Psi(f, E)$ holds.
\end{definition}

By extension, if $(\Psf_k)_{k \in \NN}$ is a family of $\RT$-like statements, then we say that a finite set~$F$ is \emph{$\Psf$-$\L$-large} if it is $\Psf_k$-$\L$-large for every~$k \leq \min F$. For instance, a set $F$ is $\RT^1$-$\L$-large if for every~$k \leq \min F$ and every coloring $f : \NN \to k$, there is an $\L$-large $f$-homogeneous subset of~$F$. Equivalently, since $\L$ is closed under supersets, $F$ is $\RT^1$-$\L$-large if for every~$k \leq \min F$ and every $k$-partition $E_0 \sqcup \dots \sqcup E_{k-1} = F$, there is some~$i < k$ such that $E_i$ is $\L$-large.
\bigskip

\noindent
\textbf{Sparsity.}
Every notion of largeness induces a dual notion of sparsity.

\begin{definition}
Let $\L \subseteq [\NN]^{<\NN}$ be a largeness notion.
A set $X \subseteq \NN$ is \emph{$\L$-sparse} if for every~$x, y \in X$ such that $x < y$,
$(x, y]_\NN \in \L$. 
\end{definition}

By closure of largeness notions under supersets, any subset of an $\L$-sparse set is again $\L$-sparse.
Any increasing function $g : \NN \to \NN$ induces a notion of largeness $\L_g = \{ F : g(\min F-1) < \max F \}$.
Then, an $\L_g$-sparse set $X$ is such that for $x, y \in X$ with $x < y$, $g(x) < y$. We simply say that $X$ is \emph{$g$-sparse}.
We will be particularly interested in $g$-sparsity for primitive recursive functions~$g$. Indeed, by adaptations of Ketonen and Solovay~\cite{ketonen1981rapidly}, one can assume that the sets are $g$-sparse for any fixed primitive recursive function~$g$, with only a constant exponent overhead in the computations. A set $X$ is \emph{exp-sparse} if it is $(x \mapsto 4^x)$-sparse, in other words, every~$x, y \in X$ such that $x < y$, then $4^x < y$.

Although sparsity can be translated into some largeness assumption, it is convenient to separate the largeness and sparsity hypothesis in the ordinal computation of largeness, as sparsity has a better behavior with respect to composition of $\RT$-like statements and avoids overheads redundancy.

\subsection{$\alpha$-largeness}\label[section]{sect:survey-regular-largeness}

Ketonen and Solovay~\cite{ketonen1981rapidly} defined a notion of largeness called $\alpha$-largeness for every~$\alpha < \epsilon_0$ to extend the Paris-Harrington theorem. It was then used to prove $\forall \Sigma^0_2$-conservation theorems for combinatorial theorems over~$\RCA_0$~\cite{towsner2024erdos,patey2018proof,kolo2020some}. A slight variant of this notion of largeness based on the Hardy hierarchy of fast-growing functions~\cite{hardy1904theorem} was extensively studied~\cite{ratajczyk1988combinatorial,kotlarski2019model,bigorajska2006some,kotlarski2007more,bigorajska2002combinatorics,smet2010partitioning}. 

Ketonen and Solovay's $\alpha$-largeness was originally defined for every $\alpha < \epsilon_0$. However, the parameterized variants presented in this article are only defined for $\alpha$ of the form $\bbomega^n \cdot k$ with $n \geq 0$ and $k \geq 1$. To uniformize the presentation, we shall therefore also restrict the definition of Ketonen and Solvay's notion of largeness to ordinals of this form. The definition below is different, but equivalent to the original one~\cite{ketonen1981rapidly}. 

\begin{definition}\label[definition]{defi:largeness-rca0}
A set~$F \finsub \NN$ is 
\begin{itemize}
    \item \emph{$\bbomega^0$-large} if $F \neq \emptyset$
    \item \emph{$\bbomega^{(n+1)}$-large} if $F \setminus \min F$ is  $(\bbomega^n \cdot \min F)$-large
    \item \emph{$\bbomega^n \cdot k$-large} if
there are $k$ $\bbomega^n$-large subsets of~$F$
$$
F_0 < F_1 < \dots < F_{k-1}
$$
where $A < B$ means that for all $a \in A$ and $b \in B$, $a < b$.
\end{itemize}
\end{definition}

Let $\Jcut$ be the $\forall \Sigma_2$-set of all~$n \in \NN$ such that $\bbomega^n$-largeness is a notion of largeness, that is, such that every infinite set contains an $\bbomega^n$-large subset.
$\RCA_0$ proves that $\Jcut$ is an additive cut, but nothing more, in the sense that there is a model~$\M = (M, S)$ of~$\RCA_0$ in which $\Jcut$ is of the form $\sup(a \cdot n : n \in \omega)$ for some non-standard element~$a \in M \setminus \omega$. Throughout this article, we shall always assume that $\min F \geq 3$ to avoid some pathological cases.
Ketonen and Solovay~\cite[Lemma 6.3]{ketonen1981rapidly} computed a first upper bound on $\RT^2_k$-$\bbomega$-largeness.

\begin{theorem}[Ketonen and Solovay~\cite{ketonen1981rapidly}]
Let $k \geq 2$. If $F \finsub \NN$ is $\bbomega^{k+4}$-large, then it is $\RT^2_k$-$\bbomega$-large.
\end{theorem}

Ko{\l}odziejczyk and Yokoyama~\cite{kolo2020some} gave an alternative proof, which emphasizes the separation between largeness and sparsity.\footnote{Ko{\l}odziejczyk and Yokoyama~\cite[Theorem 5.1]{kolo2020some} actually proved that if $F$ is $\bbomega^k+1$-large, $(x \mapsto 2x^{x^2})$-sparse and $\min F > k$, then it is $\RT^2_k$-$\bbomega$-large. If a set is $\bbomega^k \cdot 2$-large with $k \geq 2$, then it contains an $\bbomega^k+1$-large subset~$F$ such that $\min F > k$.}

\begin{theorem}[Ko{\l}odziejczyk and Yokoyama~\cite{kolo2020some}]\label[theorem]{thm:better-ketonen-solovay-bound}
Let $k \geq 2$. If $F \finsub \NN$ is $\bbomega^k \cdot 2$-large and $(x \mapsto 2x^{x^2})$-sparse, then it is $\RT^2_k$-$\bbomega$-large.
\end{theorem}


This bound is tight, in the sense that Kotlarski et al.~\cite[Theorem 5.4]{kotlarski2007more} proved that if a set is $\RT^2_k$-$\bbomega^n$-large, then it is $\bbomega^{kn-1}$-large\footnote{Kotlarski et al.~\cite{kotlarski2007more} used a slightly different notion of largeness, say 
$\alpha$-largeness${}^\dagger$, and proved that every $\RT^2_k$-$\bbomega^n$-large${}^\dagger$ set is $\bbomega^{kn-1}$-large${}^\dagger$. Both notions translate as follows: for $n \geq 1$, every $\bbomega^n$-large set is $\bbomega^n$-large${}^\dagger$ and every $\bbomega^n$-large${}^\dagger$ set is $\bbomega^{n-1}$-large.}. 
Patey and Yokoyama \cite{patey2018proof} proved that for every~$k \in \NN$, there is some~$n \in \NN$ such that $\RCA_0$ proves that every $\bbomega^n$-large set is $\RT^2_2$-$\bbomega^k$-large, and used this fact to deduce that $\RCA_0 + \RT^2_2$ is $\forall \Sigma^0_2$-conservative over~$\RCA_0$. Ko{\l}odziejczyk and Yokoyama~\cite{kolo2020some} made some explicit bound computation for $\RT^2_2$ and other $\RT^n_k$-like statements.

\begin{theorem}[Ko{\l}odziejczyk and Yokoyama~\cite{kolo2020some}]\label[theorem]{thm:kolo-bound-rt22-omegan-large}
Let $n \geq 2$. If $F \finsub \NN$ is $\bbomega^{144(n+1)}$-large and exp-sparse, then it is $\RT^2_2$-$\bbomega^n$-large.
\end{theorem}

Ko{\l}odziejczyk and Yokoyama~\cite[Lemma 2.1]{kolo2020some} showed that if $X$ is $\bbomega^{n+4}$-large then it contains an $\bbomega^n$-large exp-sparse subset.
In the same way that Ramsey's theorem for pairs is traditionally proven inductively using the infinite pigeonhole principle, Ko{\l}odziejczyk and Yokoyama~\cite{kolo2020some} based their computation of \Cref{thm:kolo-bound-rt22-omegan-large} on the following lemma:

\begin{lemma}[Ko{\l}odziejczyk and Yokoyama~\cite{kolo2020some}]
If $F \finsub \NN$ is $\bbomega^{n+1}$-large and exp-sparse, then it is $\RT^1$-$\bbomega^n$-large.
\end{lemma}

Ko{\l}odziejczyk, Wong and Yokoyama~\cite{kolodziejczyk2023ramsey} used the polynomial nature of the bound for $\RT^2_2$-$\bbomega^n$-largeness to prove that $\RCA_0 + \RT^2_2$ is polynomially simulated by $\RCA_0$ with respect to $\forall \Sigma^0_2$-sentences. In an incoming paper, Le Houérou and Patey~\cite{houerou2026partition} give a tight upper bound to $\RT^2_k$-$\bbomega^n$-largeness, up to an additive constant: In what follows, let $R_k(d)$ be the least number~$R$ such that for every coloring $f : [R]^2 \to k$, there is an $f$-homogeneous set of size~$d$.

\begin{theorem}[Le Houérou and Patey~\cite{houerou2026partition}]
Let $n, k \geq 1$. If $X \finsub \NN$ is $\bbomega^{kn} \cdot 4$-large and $(x \mapsto x^{R_{x}(2x + 2)})$-sparse with $kn \leq \min X$, then it is $\RT^2_k$-$\bbomega^n$-large.
\end{theorem}

Based on the known upper bounds of the Ramsey numbers, the authors deduced that if $X \finsub \NN$ is $\bbomega^{kn+3}$-large and $\min X \geq 17$, then it is $\RT^2_k$-$\bbomega^n$-large. 

\subsection{$\alpha$-largeness$(\theta)$}

Le Houérou, Levy Patey and Yokoyama~\cite{houerou2026pi} generalized the notion of $\alpha$-largeness to a parameterized version, in order to prove conservation theorems for $\forall \Sigma^0_3$-sentences. They used this framework to prove that $\RCA_0 + \RT^2_2$ is $\forall \Sigma^0_3$-conserva\-tive over~$\RCA_0 + \BSig_2$, and Le Houérou and Levy Patey~\cite{houerou2023conservation} proved a similar conservation result for the Ordered Variable Word theorem.
In what follows, fix a $\Delta^0_0$-formula $\theta(x, y, z)$, and let $T \equiv \forall x \exists y \forall z \theta(x, y, z)$.

\begin{definition}
Two finite sets~$E < F$ are \emph{$\theta$-apart} if 
$$\forall x < \max E\ \exists y < \min F\ \forall z < \max F\ \theta(x, y, z)$$
\end{definition}

Note that $\theta$-apartness is a transitive relation. Moreover, if $X < Y$ are $\theta$-apart and $X_0 \subseteq X$ and $Y_0 \subseteq Y$, then $X_0, Y_0$ are $\theta$-apart. The following inductive definition is very similar to that of \Cref{defi:largeness-rca0}, except that the sets $F_0, \dots, F_{k-1}$ are required to be pairwise $\theta$-apart.

\begin{definition}\label[definition]{defi:largeness-rca0-bsig2-t}
A set~$F \finsub \NN$ is 
\begin{itemize}
    \item \emph{$\bbomega^0$-large$(\theta)$} if $F \neq \emptyset$
    \item \emph{$\bbomega^{(n+1)}$-large$(\theta)$} if $F \setminus \min F$ is  $(\bbomega^n \cdot \min F)$-large$(\theta)$
    \item \emph{$\bbomega^n \cdot k$-large$(\theta)$} if
there are $k$ pairwise $\theta$-apart $\bbomega^n$-large$(\theta)$ subsets of~$F$
$$
F_0 < F_1 < \dots < F_{k-1}
$$
\end{itemize}
\end{definition}

Let $\Jcut_\theta$ be the $\forall \Sigma_2$ set of all $n \in \NN$ such that $\bbomega^n$-largeness$(\theta)$ is a notion of largeness. As in the previous largeness situation, Le Houérou et al.~\cite{houerou2026pi} showed that $\RCA_0 + \BSig_2 + T$ proves that $\Jcut_\theta$ is an additive cut, and constructed (unpublished) a model $\M = (M, S)$ of $\RCA_0 + \BSig_2 + T$ in which $J_\theta$ is again of the form $\sup(a \cdot n : n \in \omega)$ for some~$a \in M \setminus \omega$. Following the work of Ko{\l}odziejczyk and Yokoyama~\cite{kolo2020some}, the authors obtained an explicit bound for $\RT^2_2$-$\bbomega^n$-largeness$(\theta)$, except that it is exponential.

\begin{theorem}[Le Houérou, Levy Patey and Yokoyama~\cite{houerou2026pi}]
If $F \finsub \NN$ is $\bbomega^{(16^6 + 1)^{4n+4}}$-large$(\theta)$, then it is $\RT^2_2$-$\bbomega^n$-large$(\theta)$.
\end{theorem}

Thanks to the similarity of the notions of largeness and largeness$(\theta)$, most of the constructions of~\cite{kolo2020some} work mutatis mutandis, but the bound for $\RT^1$-$\bbomega^n$-largeness$(\theta)$ is $2n$ rather than $n+1$.

\begin{lemma}[Le Houérou, Levy Patey and Yokoyama~\cite{houerou2026pi}]
If $F \finsub \NN$ is $\bbomega^{2n}$-large$(\theta)$ and exp-sparse, then it is $\RT^1$-$\bbomega^n$-large$(\theta)$.
\end{lemma}

Furthermore, the authors showed that this bound is tight, in the sense that for every $n \in \NN$, there is a set~$F$ and a formula $\theta_F$ depending on~$F$ such that $F$ is $\bbomega^{2n-1}$-large$(\theta_F)$ but not $\RT^1_2$-$\bbomega^n$-large$(\theta_F)$.

\subsection{$\alpha$-largeness${}^*(\theta)$}\label[section]{subsec:new-largeness}

We now define a variant of $\alpha$-largeness$(\theta)$, namely, $\alpha$-largeness${}^*(\theta)$, for which one can prove a polynomial bound for $\RT^2_2$-$\bbomega^n$-largeness${}^*(\theta)$, opening the door to a polynomial simulation of $\RCA_0 + \RT^2_2$ over $\RCA_0 + \BSig_2$ for $\forall \Sigma^0_3$-sentences. As mentioned before, most of the combinatorics of~\cite{kolo2020some} apply to largeness$(\theta)$, except the bound for the pigeonhole principle. The goal is therefore to define a variant of largeness$(\theta)$ for which the bound for $\RT^1$-$\bbomega^n$-largeness${}^*(\theta)$ is in $n +O(1)$. The issue in adapting the proof of \cite[Lemma 2.2]{kolo2020some} to largeness$(\theta)$ is that $\theta$-apartness is not preserved by considering union of some pairwise $\theta$-apart sets. On the other hand, defining $\bbomega^n \cdot k$-largeness$(\theta)$ as $k$ many $\bbomega^n$-large$(\theta)$ subsets of~$F$ such that for every~$s < k-1$, $F_0 \cup \cdots \cup F_s$ and $F_{s+1} \cup \dots \cup F_{k-1}$ are $\theta$-apart, does not seem to yield a notion of largeness provably in $\BSig_2$. We therefore adopt a hybrid approach, which allows to preserve $\theta$-apartness under a fixed amount of union operations.

\begin{definition}
Let $\L \subseteq [\NN]^{<\NN}$ be a notion of largeness.
\begin{itemize}
    \item[(a)] Given $k \in \NN$, we write $\L \cdot k$ for the collection of all finite sets $F$ containing $k$ pairwise $\theta$-apart $\L$-large subsets $X_0 < \dots < X_{k-1} \subseteq F$.
    \item[(b)] Given $\sigma \in \NN^{<\NN}$, we write $\L \cdot \sigma$ for the collection of finite sets defined inductively as $\L \cdot \epsilon = \L$ and $\L \cdot (k \cdot \tau) = (\L \cdot k) \cdot \tau$.
\end{itemize}
\end{definition}

One can think of $\L \cdot (k, \shortbrace{d}, k)$ as a $k$-regular tree of depth~$d$ whose leaves are pairwise disjoint $\L$-large sets, with some $\theta$-apartness property at every node of the tree. This intuition is formalized in \Cref{lem:product-is-tree}.

\begin{definition}
Given $n$, $L^\theta_n$ is defined inductively as follows.
$L^\theta_0$ is the collection of all non-empty sets.
Having defined $L^\theta_n$, $L^\theta_{n+1}$ is the collection of all $X$ such that 
$$X \setminus \{\min X\} \in L^\theta_n \cdot (\min X, \shortbrace{n+1}, \min X)$$
We say that $X$ is $\bbomega^n$-large${}^*(\theta)$ if $X \in L^\theta_n$ and $\bbomega^n \cdot \sigma$-large${}^*(\theta)$ if $X \in L^\theta_n \cdot \sigma$.
\end{definition}

Note that if one replaces $L^\theta_n \cdot (\min X, \shortbrace{n+1}, \min X)$ by $L^\theta_n \cdot \min X$ in the definition above, then one recovers the notion of largeness$(\theta)$. Thus, any $\bbomega^n$-large${}^*(\theta)$ set is $\bbomega^n$-large$(\theta)$, and a fortiori $\bbomega^n$-large. A reversal is studied in \Cref{sec:new-largeness-link-to-old}. We shall prove in \Cref{sec:new-largeness-rt1} that the bound for 
$\RT^1$-$\bbomega^n$-largeness${}^*(\theta)$ is $n+2$, which is sufficient for our purposes.
The following theorem is one of the main results of this article:

\begin{theorem}
There is a primitive recursive function~$g : \NN \to \NN$ such that if $F \finsub \NN$ is $\bbomega^{12n^2+12n+6}$-large${}^*(\theta)$ and $g$-sparse, then it is $\RT^2_2$-$\bbomega^n$-large${}^*(\theta)$.
\end{theorem}

Using the translations between largeness$(\theta)$ and largeness${}^*(\theta)$ studied in \Cref{sec:new-largeness-link-to-old}, this bounds translates into a polynomial bound for $\RT^2_2$-largeness$(\theta)$.

\begin{figure}[htbp]
\begin{center}
\def\arraystretch{1.5}%
\begin{tabular}{|c|c|c|c|c|}
\hline
$\Psf$ & $\Psf$-$\bbomega^n$-large & $\Psf$-$\bbomega^n$-large$(\theta)$ & $\Psf$-$\bbomega^n$-large${}^*(\theta)$\\
\hline
$\RT^1_2$ & $\bbomega^n \cdot 2$ \cite{kolo2020some} & $\bbomega^{2n}$ \cite{houerou2026pi} &  $\bbomega^{n+1} \cdot 2$ (\ref{cor:largeness-partition-simple}) \\ \hline
$\RT^1$ & $\bbomega^{n+1}$ \cite{kolo2020some} & $\bbomega^{2n}$ \cite{houerou2026pi} & $\bbomega^{n+2}$ (\ref{cor:largeness-partition-verysimple}) \\ \hline
$\EM$ & $\bbomega^{36n}$ \cite{kolo2020some} \hspace{10pt} $\bbomega^{n} \cdot 2$ \cite{houerou2026partition} & $\bbomega^{(16^6+1)^n}$ \cite{houerou2026pi} & $\bbomega^{6n}$ (\ref{thm:em-largeness-star-theta-bound}) \\ \hline
$\ADS$ & $\bbomega^{4n+4}$ \cite{kolo2020some} \hspace{10pt} $\bbomega^{2n} \cdot 2$ \cite{houerou2026partition} & $\bbomega^{4n+4}$ \cite{houerou2026pi} & $\bbomega^{2n^2+2n} \cdot 2$ (\ref{cor:ads-new-largeness}) \\ \hline
$\RT^2_2$ & $\bbomega^{144(n+1)}$ \cite{kolo2020some} \hspace{10pt} $\bbomega^{2n} \cdot 4$ \cite{houerou2026partition} & $\bbomega^{(16^6+1)^{4n+4}}$ \cite{houerou2026pi} & $\bbomega^{12n^2+12n+6}$ (\ref{thm:rt22-largeness-star-theta-bound}) \\
\hline
\end{tabular}
\caption{Upper bounds of $\Psf$-$\bbomega^n$-largeness for various notions of largeness and various $\RT$-like statements, under the assumptions of $g$-sparsity for some primitive recursive function~$g$ and of sufficiently large minimum.}
\end{center}
\end{figure}

\section{Combinatorics of largeness${}^*(\theta)$}\label[section]{sec:combinatorics}

In this section, we develop the framework of $\alpha$-largeness${}^*(\theta)$. It admits many similarities with the development of Le Houérou, Levy Patey and Yokoyama~\cite{houerou2026pi}, with better upper bounds. However, the proofs and notation a significantly more complicated due to the structural differences of the two notions of largeness.

We start with a technical lemma which formalizes the intuition that $\L \cdot (k, \shortbrace{d} k)$ admits a tree structure, with some $\theta$-apartness conditions at each node. More precisely, every set in $\L \cdot (k, \shortbrace{d} k)$ can be represented as a $k$-regular tree of depth~$d$ whose leaves are pairwise disjoint $\L$-large sets, with some $\theta$-apartness property at every node of the tree.

\begin{lemma}\label[lemma]{lem:product-is-tree}
Let $\L$ be a notion of largeness. A set $X$ is in $\L \cdot (k, \shortbrace{d} k)$ iff
there is a map $g : k^{\leq d} \to \L \cap [X]^{<\NN}$ such that
\begin{itemize}
    \item[(a)] for every~$\sigma \in k^{<d}$, $g(\sigma) = \bigcup_{i < k} g(\sigma \cdot i)$
    \item[(b)] for every~$\ell \leq d$ and $\sigma, \tau \in k^\ell$ with $\sigma <_{lex} \tau$, then $g(\sigma) < g(\tau)$ and $g(\sigma)$ and $g(\tau)$ are $\theta$-apart
\end{itemize}
\end{lemma}
\begin{proof}
By induction on $d$. \\

For $d = 0$, let $X$ be in $\L \cdot (k, \shortbrace{0}, k) = \L$, and let $g : k^{\leq 0} \to [X]^{<\NN}$ defined by $g(\epsilon) = X$, $g$ vacuously satisfy $(a)$ and $(b)$. Reciprocally, given a function $g : k^{\leq 0} \to [X]^{<\NN}$ satisfying $(a)$ and $(b)$, $g(\epsilon)$ belongs to $\L \cap [X]^{<\NN}$, hence $X$ is in $\L = \L \cdot (k, \shortbrace{0}, k)$. \\

Assume the property to be true for some $d \in \NN$. 

Let $X$ be in $\L \cdot (k, \shortbrace{d+1}, k)$, there exists pairwise $\theta$-apart subsets $X_0 < \dots < X_{k-1} \subseteq X$ such that $X_i \in \L \cdot (k, \shortbrace{d}, k)$ for every $i < k$. By the inductive hypothesis applied to every $X_i$, there exists maps $g_i : k^{\leq d} \to \L \cap [X_i]^{<\NN}$ satisfying $(a)$ and $(b)$ for every $i < k$. Let $g : k^{\leq (d+1)} \to \L \cap [X]^{<\NN}$ defined by $g(\epsilon) = g(0) \cup \dots \cup g(k-1)$ and $g(i \cdot \sigma) = g_i(\sigma)$ for every $\sigma \in k^{\leq d}$ and $i < k$. It is clear that $g$ satisfies $(a)$ and $(b)$.

Reciprocally, consider such a function $g : k^{\leq (d+1)} \to \L \cap [X]^{<\NN}$ satisfying $(a)$ and $(b)$. For $i < k$, the map $g_i : k^{\leq d} \to \L \cap [g(i)]^{<\NN}$ defined by $g_i(\sigma) = g(i \cdot \sigma)$ also satisfy $(a)$ and $(b)$, thus, by the inductive hypothesis, the sets $g(0), \dots, g(k-1)$ are in $\L \cdot (k, \shortbrace{d},k)$, and by $(b)$, they satisfy $g(0) < \dots < g(k-1)$ and are pairwise $\theta$-apart. Therefore, the set $g(0) \cup \dots \cup g(k-1)$ is in $\L \cdot (k, \shortbrace{d+1},k)$ and, as $g(0) \cup \dots \cup g(k-1) \subseteq X$, $X$ is also in $\L \cdot (k, \shortbrace{d+1},k)$.
\end{proof}

\subsection{Construction and deconstruction}

We now prove the standard construction and deconstruction propositions for this new notion of largeness. They intuitively state sufficient conditions for a large union of $\bbomega^{b}$-large sets to be $\bbomega^{a+b}$-large, and conversely, the amount of largeness necessary to obtain an $\bbomega^a$-large collection of $\bbomega^b$-large sets.

There are multiple possible interpretations of \qt{large collection} of sets. 
In the weakest sense, $\{ \max X_i : i < k \}$ is large, and in the strongest sense, any set $H \in \prod_{i < k} [\min X_i, \max X_i]$ is large. We shall prove the construction proposition (\Cref{prop:new-large-set-of-large-sets-is-large}) with the weakest hypothesis, and the deconstruction proposition (\Cref{prop:split-new-largeness}) with the strongest conclusion.
Note that in the case of the notion of largeness of Ketonen and Solovay, both interpretations are equivalent, by regularity: indeed, if $E = \{ x_0 < \dots < x_{k-1} \}$ is $\alpha$-large, and $F = \{ y_0 < \dots < y_{k-1} \}$ is such that $y_i \leq x_i$ for every~$i < k$, then $F$ is also $\alpha$-large.

\begin{proposition}[Construction, $\ISig_1$]\label[proposition]{prop:new-large-set-of-large-sets-is-large}
    Let $a,b \in \NN$.  Let $X_0 < X_1 < \dots < X_{k-1}$ be pairwise $\theta$-apart $\bbomega^{a}$-large${}^*(\theta)$ sets such that $\{\max X_s : s < k \}$ is $\bbomega^{\lfloor \frac{a}{2}\rfloor + 2b}$-large${}^*(\theta)$. 
    Then $\bigcup_{s < k} X_s$ is $\bbomega^{a + b}$-large${}^*(\theta)$.
\end{proposition}

\begin{proof}
By induction on $b$. The case $b = 0$ is trivial.

Assume the property to be true for some $b \in \NN$. And let $X_0 < X_1 < \dots < X_{k-1}$ be pairwise $\theta$-apart $\bbomega^{a}$-large${}^*(\theta)$ sets such that $\{\max X_s : s < k \}$ is $\bbomega^{\lfloor \frac{a}{2}\rfloor + 2(b+1)}$-large${}^*(\theta)$.
Unfolding twice the definition, $\{\max X_s : 1 \leq s < k \}$ is $$\bbomega^{\lfloor\frac{a}{2}\rfloor + 2b}\cdot (\max X_0, \shortbrace{\lfloor \frac{a}{2}\rfloor + 2b + 2 + \lfloor \frac{a}{2}\rfloor + 2b + 1}, \max X_0)\mbox{-large}^*(\theta)$$

Notice that $\lfloor \frac{a}{2}\rfloor + 2b + 2 + \lfloor \frac{a}{2}\rfloor  + 2b + 1 > a+b$.

Let $Z$ be any $\bbomega^{\lfloor \frac{a}{2}\rfloor + 2b}$-large${}^*(\theta)$ set corresponding to a leaf of the tree induced by \Cref{lem:product-is-tree}. By induction hypothesis, the set $W = \{ X_s : \max X_s \in Z \}$ is $\bbomega^{a+b}$-large${}^*(\theta)$. Then $\bigcup_{s < k} X_s$ is $\bbomega^{a+b+1}$-large${}^*(\theta)$.
\end{proof}

In the following proposition, we abuse the product notation, and see $H$ as a subset of $\NN$ intersecting every interval $[\min X_i, \max X_i]$ exactly once rather than a tuple in the product $[\min X_0, \max X_0] \times \dots \times [\min X_{k-1}, \max X_{k-1}]$. Note that in the following deconstruction lemma, the bound of largeness${}^*(\theta)$ is similar to Ketonen and Solovay's largeness, namely, $\bbomega^{n+m} \cdot 2$, while for largeness$(\theta)$, one required $\bbomega^{n+2m} \cdot 2$ (see \cite[Lemma 2.5]{houerou2023conservation}).

\begin{proposition}[Deconstruction, $\ISig_1$]\label[proposition]{prop:split-new-largeness}
For every $n,m$ and for every~$\bbomega^{n+m}\cdot 2$-large${}^*(\theta)$ set~$X$, there are some $k \in \NN$ and some $\bbomega^n$-large${}^*(\theta)$ pairwise $\theta$-apart subsets $X_0 < \dots < X_{k-1}$ of~$X$ such that every $H \in \prod_{i < k} [\min X_i, \max X_i]$ is $\bbomega^m$-large${}^*(\theta)$.
\end{proposition}

\begin{proof}
The proof is similar \cite[Lemma 2.5]{houerou2023conservation}, but adapted for this new notion of largeness.

If $n = 0$, then, for $X = \{x_0, \dots, x_{k-1}\}$ an $(\bbomega^{m}\cdot 2)$-large${}^*(\theta)$ set, we can pick $X_i = \{x_i\}$ for every $i < k$. Every $X_i$ is $\bbomega^0$-large${}^*(\theta)$ and every $H \in \Pi_{i < k} [\min X_i, \max X_i]$ is equal to $X$ and is therefore $\bbomega^{m}$-large${}^*(\theta)$.

Thus, we can assume that $n > 0$, and prove the result by induction on $m$.

Case $m = 0$. The result is clear, as every non-empty set with an element greater than $3$ is $\bbomega^0$-large${}^*(\theta)$

Case $m > 0$. Let $X = Y_0 \sqcup Y_1$ for $Y_0 < Y_1$ two $\bbomega^{n+m}$-large${}^*(\theta)$ and $\theta$-apart sets. $Y_1$ is $\bbomega^{n+m-1}\cdot (\max Y_0, \shortbrace{m+n}, \max Y_0)$-large${}^*(\theta)$, and, as $n > 0$, it is therefore $\bbomega^{n+m-1}\cdot (\max Y_0, \shortbrace{m+1}, \max Y_0)$-large${}^*(\theta)$. Let $(Z_{\sigma})_{\sigma \in (\max Y_0)^{m+1}}$ be the $\bbomega^{n+m-1}$-large${}^*(\theta)$ subsets of $Y_1$ corresponding to the leaves of the tree induced by \Cref{lem:product-is-tree}. We will only use the sets $Z_{\tau \cdot i}$ for $\tau \in (\max Y_0)^{m}$ and $i < 2$.

For $\tau \in (\max Y_0)^{m}$, we can apply the inductive hypothesis on the set $Z_{\tau \cdot 0} \cup Z_{\tau \cdot 1}$ and obtain some sequence $X_0^{\tau} < \dots < X_{k_{\tau} - 1}^{\tau}$ of $\theta$-apart, $\bbomega^n$-large${}^*(\theta)$ subsets of $Z_{\tau \cdot 0} \cup Z_{\tau \cdot 1}$ such that every $H^{\tau} \in \Pi_{j < k_{\tau}} [\min X_j^{\tau}, \max X_j^{\tau}]$ is $\bbomega^{m-1}$-large${}^*(\theta)$.

Consider the family comprised of $Y_0$ and all the $X_j^{\tau}$ for $\tau \in (\max Y_0)^{m}$ and $j < k_{\tau}$. Let $k$ be the cardinality of the family and write $W_i$ for its $i$-th element. 

Every $W_i$ is $\bbomega^n$-large${}^*(\theta)$ and they are pairwise $\theta$-apart. Indeed, $Y_0$ is $\theta$-apart from the other $W_i$ as they are subsets of $Y_1$ and $Y_0$ is $\theta$-apart from $Y_1$ by assumption. For $\tau, \tau' \in (\max Y_0)^{m}$, $j < k_{\tau}$ and $j' < k_{\tau'}$, $X_j^{\tau}$ and $X_{j'}^{\tau'}$ are $\theta$-apart, indeed, if $\tau = \tau'$, this is true by the inductive hypothesis, and if $\tau \neq \tau'$, we can assume without any loss of generality that $\tau <_{lex} \tau'$, and by \Cref{lem:product-is-tree}, $\bigcup_{i < \max Y_0} Z_{\tau \cdot i}$ is $\theta$-apart from $\bigcup_{i < \max Y_0} Z_{\tau' \cdot i}$, hence so are $X_j^{\tau}$ and $X_{j'}^{\tau'}$ as they are subsets of these unions. 

Finally, every $H \in \Pi_{i < k} [\min W_i, \max W_i]$ is $\bbomega^m$-large${}^*(\theta)$. Indeed, $\max W_0 = \max Y_0$, thus $\min H \leq \max Y_0$ and every $H^0 \in \Pi_{0 < i < k} [\min W_i, \max W_i]$ is $\bbomega^{m-1}\cdot (\max Y_0, \shortbrace{m}, \max Y_0)$-large${}^*(\theta)$ as every $H^{\tau} \in \Pi_{j < k_{\tau}} [\min X_j^{\tau}, \max X_j^{\tau}]$ is $\bbomega^{m-1}$-large${}^*(\theta)$ and is a subset of $\bigcup_{i < \max Y_0} Z_{\tau \cdot i}$ for $\tau \in (\max Y_0)^{m}$.
\end{proof}


\begin{corollary}\label[corollary]{cor:split-new-largeness-gen}
For every $n,m, s$ and for every~$\bbomega^{n+m}\cdot (2, k_1, \dots, k_s)$-large${}^*(\theta)$ set~$X$, there are some $k \in \NN$ and some $\bbomega^n$-large${}^*(\theta)$ pairwise $\theta$-apart subsets $X_0 < \dots < X_{k-1}$ of~$X$ such that every $H \in \prod_{i < k} [\min X_i, \max X_i]$ is $\bbomega^m \cdot (k_1, \dots, k_s)$-large${}^*(\theta)$.
\end{corollary}
\begin{proof}
Immediate by \Cref{prop:split-new-largeness} and \Cref{lem:product-is-tree}. Indeed, any $\bbomega^{n+m}\cdot (2, k_1, \dots, k_s)$-large${}^*(\theta)$ set is $(\bbomega^{n+m}\cdot 2) \cdot (k_1, \dots, k_s)$-large${}^*(\theta)$.
\end{proof}

\subsection{Relation to largeness$(\theta)$}\label[section]{sec:new-largeness-link-to-old}

As mentioned in \Cref{subsec:new-largeness}, any $\bbomega^n$-large${}^*(\theta)$ set is $\bbomega^n$-large$(\theta)$. We now prove a partial reversal, which will be useful in \Cref{subsect:rt22-new-largeness} to translate existing theorems over largeness$(\theta)$ to theorems over largeness${}^*(\theta)$. However, the non-optimality of the translation makes it necessary to reprove a significant part of the framework to obtain better bounds.


By \Cref{lem:product-is-tree}, one can see any $\bbomega^n \cdot (\min X, \shortbrace{k} , \min X)$-large$(\theta)$ as a $\min X$-regular tree of depth~$k$ whose leaves are pairwise disjoint $\bbomega^n$-large$(\theta)$ sets, with $\theta$-apartness conditions. 
Similarly, by the definition of largeness$(\theta)$, any $\bbomega^{n+k}$-large$(\theta)$ set can be seen as a tree of $\bbomega^n$-large$(\theta)$ sets, with a level of branching of the parent nodes depending on the children. This \qt{dependently branching} tree is much more branching than $\min X$. The following lemma reflects this explanation:

\begin{lemma}\label[lemma]{lem:largeness-to-products}
Fix $n \geq 0$ and $k > 0$.
If $X$ is $\bbomega^{n+k}$-large$(\theta)$, then $X \setminus \{\min X\}$ is 
$\bbomega^n \cdot (\min X, \shortbrace{k} , \min X)$-large$(\theta)$.
\end{lemma}
\begin{proof}
By induction on~$k$. 
Case $k = 1$. Then $X \setminus \{\min X\}$ is $\bbomega^n \cdot \min X$-large$(\theta)$ by definition.
Suppose $k > 1$. Then $X \setminus \{\min X\}$ is $\bbomega^{n+k-1} \cdot \min X$-large$(\theta)$.
Let $X_0 < \dots < X_{\min X-1} \subseteq X \setminus \{\min X\}$ be pairwise $\theta$-apart $\bbomega^{n+k-1}$-large$(\theta)$ sets.
By induction hypothesis, for every~$i < \min X$, letting $Y_i = X_i \setminus \{\min X_i\}$, $Y_i$ is $\bbomega^n \cdot (\min X_i, \shortbrace{k-1}, \min X_i)$-large$(\theta)$, hence is $\bbomega^n \cdot (\min X, \shortbrace{k-1}, \min X)$-large$(\theta)$. Moreover, the $Y_i$'s are pairwise $\theta$-apart, so $X \setminus \{\min X\}$ is $\bbomega^n \cdot (\min X, \shortbrace{k}, \min X)$-large$(\theta)$.
\end{proof}

We are now ready to prove a partial reversal to the fact that every $\bbomega^n$-large${}^*(\theta)$ set is $\bbomega^n$-large$(\theta)$.

\begin{proposition}\label[proposition]{prop:large-to-new-large}
If $X$ is $\bbomega^{\frac{n(n+1)}{2}}$-large$(\theta)$, then it is $\bbomega^n$-large${}^*(\theta)$.
\end{proposition}
\begin{proof}
By induction on~$n$.
Case $n = 0$. If $X$ is $\bbomega^0$-large$(\theta)$, then $X \neq \emptyset$ so $X$ is $\bbomega^0$-large${}^*(\theta)$.
Case $n > 0$. If $X$ is $\bbomega^{\frac{n(n+1)}{2}}$-large$(\theta)$, then it is $\bbomega^{\frac{(n-1)n}{2}+n}$-large$(\theta)$. By \Cref{lem:largeness-to-products}, $X \setminus \{\min X\}$ is $\bbomega^{\frac{(n-1)n}{2}} \cdot (\min X, \shortbrace{n}, \min X)$-large$(\theta)$. By induction hypothesis, $X \setminus \{\min X\}$ is $\bbomega^{n-1} \cdot (\min X, \shortbrace{n}, \min X)$-large${}^*(\theta)$, hence $X$ is $\bbomega^n$-large${}^*(\theta)$.
\end{proof}


\subsection{Cut}

Recall that a \emph{cut} is an initial segment of~$\NN$ closed under successor.
A cut~$I$ is \emph{additive} if for every~$n \in I$, $n+n \in I$.

Let $\Jcut_\theta$ (resp. $\Jcut^*_\theta$) be the set of all $n \in \NN$ such that every infinite set contains a finite subset that is $\bbomega^n$-large$(\theta)$ (resp. $\bbomega^n$-large${}^*(\theta)$). The authors showed in \cite{houerou2026pi} that $\Jcut_\theta$ is an additive cut, provably in $\RCA_0 + \BSig_2 + T$. In this section, we study the cut $\Jcut^*_\theta$, and show that it satisfies the same closure properties as $\Jcut_\theta$ over $\RCA_0 + \BSig_2 + T$, that is, it is an additive cut, and no more closure properties are provable, in that there is a model $\M = (M, S) \models \RCA_0 + \BSig_2 + T$ and some $a \in M \setminus \omega$ such that $\Jcut^*_\theta(\M) = \sup ( a \cdot n : n \in \omega )$.

\begin{lemma}[$\RCA_0 + \BSig_2 + T$]\label[lemma]{lem:infinite-pairwise-apart-large}
Let $\L$ be a largeness notion. For every infinite set~$X$,
there is an infinite sequence $E_0 < E_1 < \dots$ of pairwise $\theta$-apart $\L$-large subsets of $X$.
\end{lemma}
\begin{proof}
Consider the greedy algorithm enumerating such a sequence. 
Once $E_0< \dots< E_{k-1}$ have been listed, search for some $E_k > E_{k-1}$ that is $\theta$-apart from $E_0, \dots, E_{k-1}$ and $\L$-large. If for some $k \in \NN$, $E_0, \dots, E_{k-1}$ have been outputted by this procedure, since $T$ and $\BSig_2$ hold, there exists a bound $b_k$ such that every subset of $[b_k,+\infty)$ is $\theta$-apart from $E_0 \cup \dots \cup E_{k-1}$.
Since $X \cap [b_k, +\infty)$ is an infinite set and $\L$ is a largeness notion, there exists some $\L$-large subset $E_k \subseteq X \cap [b_k, +\infty)$, hence the greedy algorithm will eventually list some $E_k$.
\end{proof}

\begin{proposition}[$\RCA_0 + \BSig_2 + T$]\label[proposition]{prop:new-cut-is-additive}
$\Jcut_{\theta}^*$ is an additive cut.
\end{proposition}
\begin{proof}
Any $\bbomega^{n+1}$-large${}^*(\theta)$ set contains many $\bbomega^n$-large${}^*(\theta)$ subsets, so by $\Delta^0_0$-induction which holds over $\RCA_0$, $\Jcut_\theta^*$ is an initial segment of~$\NN$.

By \cite{houerou2026pi}, $\RCA_0 + \BSig_2 + T$ proves that $\Jcut_\theta$ is a cut, so $\omega \subseteq \Jcut_\theta$.
By \Cref{prop:large-to-new-large}, every $\bbomega^{\frac{n(n+1)}{2}}$-large$(\theta)$ set is $\bbomega^n$-large${}^*(\theta)$. It follows that $\omega \subseteq \Jcut^*_\theta$.

We now prove that if $n \in \Jcut_{\theta}^*$, then $n+\lfloor \frac{n}{4}\rfloor \in \Jcut_{\theta}^*$.
Suppose that $n \in \Jcut_{\theta}^*$. Let $X$ be an infinite set. 
By \Cref{lem:infinite-pairwise-apart-large}, there is a countable collection $E_0 < E_1 < \dots$ of pairwise $\theta$-apart  $\bbomega^n$-large${}^*(\theta)$ subsets of $X$.
The set $Y = \{ \max E_s : s \in \NN \}$ is infinite, so there is an $\bbomega^n$-large${}^*(\theta)$ subset~$Z \subseteq X$.
By \Cref{prop:new-large-set-of-large-sets-is-large} with $a = n$ and $b = \lfloor \frac{n}{4}\rfloor$,
$\bigcup_{s \in Z} E_s$ is an $\bbomega^{n+\lfloor \frac{n}{4}\rfloor}$-large${}^*(\theta)$ subset of~$X$, so $n + \lfloor \frac{n}{4}\rfloor \in \Jcut_{\theta}^*$.

Next, we prove that if $n \in \Jcut_{\theta}^*$, then $n+n \in \Jcut_{\theta}^*$.
Let $n \in \Jcut_{\theta}^*$. If $n \in \omega$, we are done, so suppose not. By the previous claim, $n + 5 \lfloor \frac{n}{4}\rfloor \in \Jcut_{\theta}^*$. Note that since $n \not \in \omega$ (and in particular $n \geq 16$), $5 \lfloor \frac{n}{4}\rfloor > n$, so since $\Jcut_{\theta}^*$ is an initial segment of $\NN$, $n + n \in \Jcut_{\theta}^*$.

Since $\Jcut_\theta^*$ is an initial segment of~$\NN$ closed under addition and containing~$\omega$, it is also closed under successor, so $\Jcut_\theta^*$ is an additive cut.
\end{proof}

The remainder of this section is mostly focused on proving that $\Jcut_{\theta}^*$ satisfies no better closure property than additivity over $\RCA_0 + \BSig_2 + T$.
A cut $I \subseteq M$ is said to be \emph{semi-regular} if for every $M$-finite set $E \subseteq M$ such that $|E| \in I$, $E \cap I$ is bounded in $I$.
The following proposition shows that semi-regular cuts are exactly those coding a model of~$\WKL_0$:

\begin{proposition}[{See Scott \cite{scott1962algebras} and Kirby and Paris \cite[Proposition 1]{Kirby1977InitialSO}}]
Let $I \subseteq M$ be a cut and let $\Cod(M/I) = \{S \cap I : S \mbox{ is } M\mbox{-finite}\}$. Then $I$ is semi-regular if and only if $(I, \Cod(M/I)) \models \WKL_0$.   
\end{proposition}

In what follows, let $\PAm$ denote the axioms of a discrete ordered commutative semiring (see \cite[Definition 4.1]{hirschfeldt2015slicing} for instance).

\begin{proposition}\label[proposition]{prop:semiregular-wf-add}
Let~$M \models \PAm + \Bsig_2 + T$ be a countable non-standard model.
There exists some~$a \in M \setminus \omega$ and a proper semi-regular cut~$I \prec_e M$ such that, letting $\M = (I, \Cod(M/I))$, $\M \models \WKL_0 + \BSig_2 + T$ and $\Jcut_\theta^*(\M) = \sup \{ a \cdot k : k \in \omega \}$.
\end{proposition}
\begin{proof}
By \Cref{prop:new-cut-is-additive}, for every $n \in \omega$, $\PAm + \Bsig_2 + T$ proves that there exists some exp-sparse $\bbomega^n$-large${}^*(\theta)$ set. Therefore, by overflow, there exists some $a \in M \setminus \omega$ and an exp-sparse $\bbomega^{a \cdot a}$-large${}^*(\theta)$ subset of $M$. Let $F_{-1}$ be such a set.

Fix $(S_n)_{n \in \omega}$ an enumeration of all the $M$-finite sets that are not $\bbomega^{a+1}$-large${}^*(\theta)$, fix $(E_i)_{i \in \omega}$ an enumeration containing all the $M$-finite sets infinitely many times, and fix $(c_n)_{n \in \omega}$ an enumeration of all the non-standard elements of $M$.
The cut $I$ will be defined as $\sup \{\min(F_n) | n \in \omega\}$ for $F_{-1} \supseteq F_0 \supseteq F_1 \dots$ a decreasing sequence of $M$-finite set satisfying for all $i \in \omega$:

\begin{enumerate}
    \item[(1)] $F_i$ is $\bbomega^{a \cdot c}$-large${}^*(\theta)$ for some non-standard $c$. 
    \item[(2)] $F_{4i}$ is not $\bbomega^{a \cdot c_i}$-large${}^*(\theta)$ and $\min F_{4i} > \min F_{4i-1}$ for all $i \in \omega$ 
    \item[(3)] $[\min F_{4i+1}, \max F_{4i+1}] \cap S_i = \emptyset$
    \item[(4)] If $|E_i| < \min F_{4i+1}$ then $[\min F_{4i+2}, \max F_{4i+2}] \cap E_i = \emptyset$
    \item[(5)] $\{\min F_{4i+2} \}$ is $\theta$-apart from $F_{4i+3}$.
\end{enumerate}

Assume $F_{n}$ has already been defined for some $n \in \omega \cup \{-1\}$.

If $n + 1 = 4m$ for some $m$: Let $F_{4m} \subseteq F_n$ be such that $F_{4m}$ is $\bbomega^{a \cdot c}$-large${}^*(\theta)$ for some non-standard $c$ but not $\bbomega^{a \cdot c_i}$-large${}^*(\theta)$, and $\min F_n < \min F_{4m}$.

If $n + 1 = 4m+1$ for some $m$: Let $c$ non-standard be such that $F_n$ is $\bbomega^{a \cdot c}$-large${}^*(\theta)$. $F_n$ being $\bbomega^{a \cdot (c-2) + (a+1)} \cdot 2$-large${}^*(\theta)$, by \Cref{prop:split-new-largeness}, there exists a family $X_0 < \dots < X_{k-1}$ of $\bbomega^{a \cdot (c-2)}$-large${}^*(\theta)$ subsets of $X$ such that every $H \in \Pi_{i < k} [\min X_i, \max X_i]$ is $\bbomega^{a+1}$-large${}^*(\theta)$. There exists some $i < k$ such that $[\min X_{i}, \max X_i] \cap S_m = \emptyset$, otherwise there would be some $H \in \Pi_{i < k} [\min X_i, \max X_i]$ such that $H \subseteq S_m$, contradicting the fact that $S_m$ is not $\bbomega^{a+1}$-large${}^*(\theta)$, so let $F_{n+1} = X_i$.

If $n+1 = 4m+2$ for some $m$: if $\min F_{4m+1} \leq |E_m|$, then keep $F_{4m+2} = F_{4m+1}$. If $\min F_{4m+1} > |E_m|$, then, let $c$ non-standard be such that $F_{4m+1}$ is $\bbomega^{a \cdot c}$-large${}^*(\theta)$ and let $X_0 < \dots < X_{\min F_{4m+1} -1}$ be $\bbomega^{a \cdot c - 1 }$-large${}^*(\theta)$ subsets of $F_{4m+1}$. By the finite pigeonhole principle, there exists a $j < \min F_{4m+1}$ such that $E_i \cap [\min X_j, \max X_j] = \emptyset$, in this case, take $F_{4m+2} = X_j$ for such a $j$.

If $n+1 = 4m+3$ for some $m$: Let $c$ non-standard be such that $F_n$ is $\bbomega^{a \cdot c}$-large${}^*(\theta)$. Thus, there exists some $\theta$-apart and $\bbomega^{a\cdot (c-1)}$-large${}^*(\theta)$ subsets $X_0 < X_1$ of $F_n$. Take $F_{4m+3} = X_1$, $F_{4m+3}$ is $\theta$-apart from $X_0$, hence $\theta$-apart from $[0, \max X_0]$, and therefore $\theta$-apart from $\{\min F_{4i+2} \}$.
\\

We first claim that $(I, \Cod(M/I)) \models \WKL_0 + \BSig_2 + T$.
The constraint $\min F_{4i} > \min F_{4i-1}$ of (2) ensures that $I$ is a cut and that every $F_i \cap I$ is cofinal in $I$. By (4), the cut will be semi-regular, which, using \cite[Theorem 7]{Kirby1977InitialSO} implies that $(I, \Cod(M/I)) \models \WKL_0$.
By Paris–Kirby~\cite{paris1978collection} or Adamowicz–Clote–Wilkie~\cite{clote1985partition}, since $M \models \ISig_1$, then $\Cod(M/I) \models \BSig_2$.
The constraint $(5)$ ensures that $(I, \Cod(M/I)) \models T$. Indeed, for every $x \in I$, there exists some $i \in \NN$ such that $x < \min F_{4i+2}$. Since $\{\min F_{4i+2}\}$ is $\theta$-apart from $F_{4i+3}$, there exists some $y < \min F_{4i+3}$ such that $\theta(x,y,z)$ holds for every $z < \max F_{4i+3}$. By construction of $I$, we have $y \in I$ and $\max F_{4i+3} > I$, thus $(I, \Cod(M/I)) \models T$.

We now claim that $\Jcut_\theta^*(\M) = \sup \{ a \cdot k : k \in \omega \}$.
By definition of $I$, $F_{4i} \cap I$ is an element of $\Cod(M/I)$ that is cofinal in $I$ and which is not $\bbomega^{a \cdot c_i}$-large${}^*(\theta)$ by~(2). For every $d >\sup \{ a \cdot k : k \in \omega \}$, there exists a non-standard $c$ such that $a \cdot c < d$ (by overflow). Then for $i \in \omega$ such that $c = c_i$, the set $F_{4i} \cap I$ does not contain any $\bbomega^{d}$-large${}^*(\theta)$ $I$-finite subset. Therefore, $\Jcut^*_\theta(\M) \subseteq \sup \{ a \cdot k : k \in \omega \}$.

For the reverse inclusion, it is sufficient to show that $a \in \Jcut^*_\theta(\M)$ (as $\Jcut^*_\theta(\M)$ is closed by summation by \Cref{prop:new-cut-is-additive}). Take any $U \in \Cod(M/I)$ cofinal in $I$, there exists some $M$-finite set $U'$ such that $U = U' \cap I$. $U'$ is $\bbomega^{a+1}$-large${}^*(\theta)$, otherwise $U'$ would be one of the $S_i$, and $U$ would not be cofinal in $I$ (as $[\min F_{4i+1}, \max F_{4i+1}] \cap S_i = \emptyset$ by (3)). Let $V$ be the smallest initial segment of $U'$ forming an $\bbomega^{a}$-large${}^*(\theta)$ set, then $V$ is not $\bbomega^{a+1}$-large${}^*(\theta)$ and therefore $V = S_j$ for some $j \in \omega$, hence $V \cap I$ is not cofinal in $I$. $U$ and $V$ are both initial segments of $U'$, and $U$ is not included in $V$ as $U$ is cofinal in $I$ unlike $V$, so $V$ is included in $U$ and therefore in $I$. So, every infinite set in $(I, \Cod(M/I))$ contains an $\bbomega^a$-large${}^*(\theta)$ $I$-finite subset, hence $a \in \Jcut^*_\theta(\M)$. 
\end{proof}

Let $\mathsf{I}^0_2$ be the intersection of all $\Sigma^0_2$ cuts, or equivalently of all $\Pi^0_2$ cuts.
It is straightforward to prove over $\RCA_0$ that $\mathsf{I}^0_2 \subseteq \Jcut$. Moreover, it can be shown that for every model $\M = (M, S) \models \RCA_0 + \BSig_2 + \neg \ISig_2$, every $A \in S$ such that $\ISig_2(A)$ fails, and~$x \in \Jcut(\M) \setminus \mathsf{I}^0_2(\M)$, there exists an $M$-unbounded set $Y \subseteq M$ such that $\M \models (Y \oplus A)' \leq_T A'$ and $Y$ contains no $\bbomega^x$-large subset; a proof will appear in a forthcoming paper by Ikari, Kołodziejczyk, and Yokoyama (personal communication). It follows that every countable model of $\RCA_0 + \BSig_2 + \neg \ISig_2$ can be $\bbomega$-extended into a model of $\RCA_0 + \BSig_2 + \neg \ISig_2 + $\qt{$\mathsf{I}^0_2 = \Jcut$}. Since $\RCA_0 + \ISig_2 \vdash \mathsf{I}^0_2 = \NN$, then $\RCA_0 + \BSig_2 + $\qt{$\mathsf{I}^0_2 = \Jcut$} is $\Pi^1_1$-conservative over $\RCA_0 + \BSig_2$.

Since every $\bbomega^n$-large${}^*(\theta)$ set is $\bbomega^n$-large$(\theta)$, which is itself $\bbomega^n$-large, then $\RCA_0 \vdash \Jcut^*_\theta \subseteq \Jcut_\theta \subseteq \Jcut$. We now prove that if one further assumes $\BSig_2 + T$, then $\mathsf{I}^0_2 \subseteq \Jcut^*_\theta$. Interestingly, the argument requires to first prove that $\mathsf{I}^0_2 \subseteq \Jcut_\theta$.

\begin{lemma}\label[lemma]{lem:i02-in-jcut-theta}
$\RCA_0 + \BSig_2 + T \vdash \mathsf{I}^0_2 \subseteq \Jcut_\theta$.
\end{lemma}
\begin{proof}
Suppose that $a \not \in \Jcut_\theta$.
Let $X$ be an infinite set with no $\bbomega^a$-large$(\theta)$ subset.
Let $I \subseteq \NN$ be the $\Pi^0_2(X)$ set of all $n \in \NN$ such that for every~$k \in \NN$,
$X \cap (k, \infty)$ contains an $\bbomega^n$-large$(\theta)$ subset. In particular, $a \not \in I$.

We prove that $I$ is a cut. First, $I$ is an initial segment of~$\NN$ as every $\bbomega^n$-large$(\theta)$ set contains $\bbomega^m$-large$(\theta)$ subsets for every~$m < n$. Moreover, $0 \in I$ since for every $k \in \NN$, $X \cap (k, \infty)$ is non-empty, hence contains an $\bbomega^0$-large$(\theta)$ set. Suppose $n \in I$. Let $k \in \NN$ and $x \in X \cap (k, \infty)$. Since $X \cap (x, \infty)$ is infinite, by the same argument as \Cref{lem:infinite-pairwise-apart-large}, there is an infinite sequence $E_0 < E_1 < \dots$ of pairwise $\theta$-apart $\bbomega^n$-large$(\theta)$ subsets of $X \cap (x, \infty)$.
Then $\{x\} \cup \bigcup_{i < x} E_i$ is an $\bbomega^{n+1}$-large$(\theta)$ subset of $X \cap (k, \infty)$. It follows that $n+1 \in I$.

Since $I$ is a $\Pi^0_2(X)$ cut, then $\mathsf{I}^0_2 \subseteq I$, so $a \not \in \mathsf{I}^0_2$.
\end{proof}

\begin{lemma}\label[lemma]{lem:large-set-of-large-sets-is-large}
    Let $\L \subseteq [\NN]^{<\NN}$ be a collection of sets and $E_0 < E_1 < \dots < E_{k-1}$ be a sequence of $\L$-large, pairwise $\theta$-apart sets such that $\{\max E_i : i < k\}$ is $\bbomega^{n+1}$-large$(\theta)$ for some $n$. Then for every $x < E_1$, $E_1 \cup \dots \cup E_{k-1} \in \L \cdot (x, \shortbrace{n}, x)$.
\end{lemma}

\begin{proof}
By induction on $n$. The case $n = 0$ is immediate. 

Assume the property to holds for some $n \in \NN$ and let $E_0 < E_1 < \dots < E_{k-1}$ be a sequence of $\L$-large, $\theta$-apart sets such that $F := \{\max E_i : i < k\}$ is $\bbomega^{n+1}$-large$(\theta)$ for some $n$ and let $x < E_0$. Let $F_0, \dots, F_{\max E_0 - 1}$ be $\theta$-apart and $\bbomega^n$-large$(\theta)$ subsets of $F$. By our inductive hypothesis, for every $i < \max E_0$, the set $Z_i := \bigcup_{\max E_j \in F_i} E_j$ is in $\L \cdot (x, \shortbrace{n}, x)$. Furthermore, the $Z_i$ are pairwise $\theta$-apart, so $E_1 \cup \dots \cup E_{k-1} \in \L \cdot (x, \shortbrace{n+1}, x)$.
\end{proof}

\begin{lemma}\label[lemma]{lem:i02-in-jcut-star-theta}
$\RCA_0 + \BSig_2 + T \vdash \mathsf{I}^0_2 \subseteq \Jcut^*_\theta$.
\end{lemma}
\begin{proof}
Suppose that $a \not \in \Jcut^*_\theta$. If $a \not \in \Jcut_\theta$, then by \Cref{lem:i02-in-jcut-theta}, $a \not \in \mathsf{I}^0_2$. From now on, assume that $a \in \Jcut_\theta$. The proof is very similar to \Cref{lem:i02-in-jcut-theta}, but with a twist.

Let $X$ be an infinite set with no $\bbomega^a$-large${}^*(\theta)$ subset.
Let $I \subseteq \NN$ be the $\Pi^0_2(X)$ set of all $n \in \NN$ such that for every~$k \in \NN$,
$X \cap (k, \infty)$ contains an $\bbomega^n$-large${}^*(\theta)$ subset. 
As in \Cref{lem:i02-in-jcut-theta}, $I$ is an initial segment of~$\NN$ containing~0 but not~$a$.

Let us prove that $I$ is closed under successor, hence is a cut.
Suppose $n \in I$. Let $k \in \NN$ and $x \in X \cap (k, \infty)$. Since $X \cap (x, \infty)$ is infinite, by the same argument as \Cref{lem:infinite-pairwise-apart-large}, there is an infinite sequence $E_0 < E_1 < \dots$ of pairwise $\theta$-apart $\bbomega^n$-large${}^*(\theta)$ subsets of $X \cap (x, \infty)$.

Let $y_i = \max E_i$ and consider the set $Y = \{y_1, y_2, \dots\}$. Since $a \in \Jcut_\theta$ and $a > n$, then $n \in \Jcut_\theta$, hence $n+2 \in \Jcut_\theta$. Thus, as the set $Y$ is infinite, there exists some subset $F \subseteq Y$ that is $\bbomega^{n+2}$-large$(\theta)$. Let $E = \{y_0\} \cup \bigcup_{y_i \in F} E_i$. By \Cref{lem:large-set-of-large-sets-is-large}, $\bigcup_{y_i \in F} E_i$ is in $L_n^{\theta} \cdot (y_0, \shortbrace{n+1}, y_0)$, hence $E$ is an $\bbomega^{n+1}$-large${}^*(\theta)$ subset of $X \cap (k, \infty)$. It follows that $n+1 \in I$.

Since $I$ is a $\Pi^0_2(X)$ cut, then $\mathsf{I}^0_2 \subseteq I$, so $a \not \in \mathsf{I}^0_2$.
\end{proof}

It follows from the above discussion that $\RCA_0 + \BSig_2 + T + \qt{\mathsf{I}^0_2 = \Jcut^*_\theta = \Jcut_\theta = \Jcut}$ is $\Pi^1_1$-conservative over $\RCA_0 + \BSig_2 + T$.

\subsection{Sparsity}\label[section]{sec:sparsity}

In Ketonen and Solovay's notion of largeness, a set~$X$ is $\bbomega^{n+1}$-large iff $X \setminus \{\min X\}$ is $\bbomega^n \cdot \min X$-large.
From a purely logical viewpoint, one could have fixed a provably total, increasing function~$g$ and defined a set~$X$ to be $\bbomega^{n+1}$-$g$-large if $X \setminus \{\min X\}$ is $\bbomega^n \cdot g(\min X)$-large. The resulting notion of largeness would enjoy all the necessary properties to prove conservation theorems, and it is often very convenient to assume that an $\bbomega^{n+1}$-large set~$X$ contains much more than $\min X$ many pairwise disjoint $\bbomega^n$-large subsets. 

Thanks to sparsity, one can switch from one definition to the other with only a multiplicative cost. For simplicity, consider for example an $\bbomega \cdot 2$-large and $g$-sparse set~$X$. Then there exist two pairwise disjoint $\bbomega$-large sets $X_0 < X_1 \subseteq X$. In particular, $g(\max X_0) < \min X_1$. By definition of $\bbomega$-largeness, $X_1 \setminus \{\min X_1\}$ is $\min X_1$-large, hence $g(\max X_0)$-large. It follows that $\{\max X_0\} \cup X_1 \setminus \{\min X_1\}$ (hence $X$) is $\bbomega$-$g$-large. Over the remainder of this article, we shall use multiple times the same idea to artificially increase the number of branches of the induced tree in the definition of an $\bbomega^{n+1}$-large${}^*(\theta)$ block.
Since we work within $\RCA_0$, the provably total computable functions are the primitive recursive ones.

\begin{theorem}[Ketonen and Solovay~\cite{ketonen1981rapidly}]\label[theorem]{thm:large-dominates-pr}
$\ISig_1$ proves that for every primitive recursive function $g : \NN \to \NN$, there is some~$n \in \NN$
such that for every~$\bbomega^n$-large set~$F$, $\max F > g(\min F)$.
\end{theorem}

In particular, $y > 2x$ whenever $(x, y]$ is $\bbomega$-large and $y > x 2^x$ whenever $(x, y]$ is $\bbomega^2$-large.
Recall that a set $X$ is \emph{$\bbomega^n$-sparse} if for every~$x, y \in X$ with $x < y$, then the interval $(x, y]$ is $\bbomega^n$-large.
Note that one could have worked with $\bbomega^n$-sparsity${}^*(\theta)$ instead of $\bbomega^n$-sparsity, but the later notion is sufficient and better understood, so in the remainder of this article, we shall consider only $\bbomega^n$-sparsity. 

\begin{corollary}[$\ISig_1$]\label[corollary]{cor:new-large-omega-k-sparse}
For every~$k, n \in \NN$, every $\bbomega^{n+k} \cdot 2$-large${}^*(\theta)$ set contains an $\bbomega^k$-sparse $\bbomega^n$-large${}^*(\theta)$ subset.
\end{corollary}
\begin{proof}
Let $X$ be $\bbomega^{n+k} \cdot 2$-large${}^*(\theta)$.
By \Cref{prop:split-new-largeness}, there are some $\ell \in \NN$ and some $\bbomega^k$-large${}^*(\theta)$ pairwise $\theta$-apart subsets $X_0 < \dots < X_{\ell-1}$ of~$X$ such that every $H \in \prod_{i < \ell} [\min X_i, \max X_i]$ is $\bbomega^n$-large${}^*(\theta)$.
In particular, the set $Y = \{ \max X_i : i < \ell \}$ is $\bbomega^n$-large${}^*(\theta)$.
Let $a, b \in Y$ with $a < b$. Let $i$ be such that $\max X_i = b$. Then $a < \min X_i$, and $X_i$ is $\bbomega^k$-large${}^*(\theta)$, so $(a, b]$ is $\bbomega^k$-large${}^*(\theta)$ hence $\bbomega^k$-large. It follows that $Y$ is $\bbomega^k$-sparse.
\end{proof}

\begin{corollary}[$\ISig_1$]\label[corollary]{cor:new-large-g-sparse}
For every primitive recursive function $g : \NN \to \NN$, there is some~$k \in \NN$ such that for every~$n \in \NN$,
every $\bbomega^{n+k}$-large${}^*(\theta)$ set contains a $g$-sparse $\bbomega^n$-large${}^*(\theta)$ subset.
\end{corollary}
\begin{proof}
Immediate by \Cref{cor:new-large-omega-k-sparse} and \Cref{thm:large-dominates-pr}.
\end{proof}

\subsection{Pigeonhole principle}\label[section]{sec:new-largeness-rt1}

We now turn to the proof of a central proposition which justifies the design of this new notion of largeness${}^*(\theta)$.
\Cref{prop:largeness-partition} gives an upper bound to the pigeonhole principle in its most general form, from which we deduce the $\bbomega^{n+2}$-large${}^*(\theta)$ upper bound to $\RT^1$-$\bbomega^n$-largeness${}^*(\theta)$ in \Cref{cor:largeness-partition-verysimple}.

\begin{proposition}\label[proposition]{prop:largeness-partition}
Let $X$ be a finite set, $n, s, k_1, \dots, k_s \in \NN$ and $a \leq \min X$.
If $X$ is $\bbomega^{n+1} \cdot (a, a k_1, \dots, a k_s)$-large${}^*(\theta)$ and exp-sparse, and $f : X \to a$ is a coloring, then there is an $f$-homogeneous $\bbomega^n \cdot (k_1, \dots, k_s)$-large${}^*(\theta)$ subset.
\end{proposition}
\begin{proof}
By induction over $n$, and for a fixed $n$, by induction over~$s$.

Case 1: $n = 0$ and $s = 0$. Any singleton element is $\bbomega^0$-large${}^*(\theta)$, so $\{\min X\}$ is $f$-homogeneous and $\bbomega^0$-large${}^*(\theta)$.

Case 2: $n > 0$ and $s = 0$. Let $X_0 < \dots < X_{a-1} \subseteq X$ be pairwise $\theta$-apart $\bbomega^{n+1}$-large${}^*(\theta)$ subsets of $X$. Suppose $X_0$ is not $f$-homogeneous, otherwise we are done.
Let~$t < a$ be maximal such that $f(\bigcup_{j < t} X_j) \supseteq f(X_t)$.
Note that $t \geq 1$ exists since otherwise, 
$$2 \leq |f(X_0)| < |f(X_0 \cup X_1)| < \dots < |f(X_0 \cup \dots \cup X_{a-1})|$$
and therefore $|f(X)| \geq a+1$, contradicting the assumption that $f : X \to a$.

Since $X_t$ is $\bbomega^{n+1}$-large${}^*(\theta)$, then $X_t \setminus \{\min X_t\}$ is $\bbomega^n \cdot (\min X_t, \shortbrace{n+1}, \min X_t)$-large${}^*(\theta)$. By exp-sparsity, $4^{\max X_{t-1}} < \min X_t$ so $(\max X_{t-1})^2 < \min X_t$, thus $a \times \max X_{t-1} < \min X_t$. Then, by induction hypothesis, there is some $\bbomega^{n-1} \cdot (\max X_{t-1}, \shortbrace{n}, \max X_{t-1})$-large${}^*(\theta)$ $f$-homogeneous subset~$Y \subseteq X_t$ for some color~$i < a$. By choice of~$t$, there is some~$c \in \bigcup_{j < t} X_j$ such that $f(c) = i$. Note that $c \leq \max X_{t-1}$, so $\{c\} \cup Y$ is $\bbomega^n$-large${}^*(\theta)$, and $f$-homogeneous for color~$i$.

Case 3: $n \geq 0$ and $s > 0$. There exists $a k_s$ pairwise $\theta$-apart $\bbomega^{n+1} \cdot (a, a k_1, \dots, a k_{s-1})$-large${}^*(\theta)$ sets $X_0 < \dots < X_{a k_s-1}$. By induction hypothesis, for each~$j < a k_s$, there is an $\bbomega^n \cdot (k_1, \dots, k_{s-1})$-large${}^*(\theta)$ $f$-homogeneous subset $Y_j \subseteq X_j$ for some color~$i_j < a$. There is some~$i < a$ such that $i_j = i$ for at least $a$ many $j$'s. Then $\bigcup_{i_j = i} Y_j$ is $f$-homogeneous for color~$i$ and $\bbomega^n \cdot (k_1, \dots, k_{s-1}, k_s)$-large${}^*(\theta)$.
\end{proof}

\begin{corollary}\label[corollary]{cor:largeness-partition-simple}
Let $X$ be a finite set, $n \in \NN$ and $a \leq \min X$.
If $X$ is $\bbomega^{n+1} \cdot a$-large${}^*(\theta)$ and exp-sparse, and $f : X \to a$ is a coloring, then there is an $f$-homogeneous $\bbomega^n$-large${}^*(\theta)$ subset.
\end{corollary}
\begin{proof}
Immediate by \Cref{prop:largeness-partition} with $s = 0$.
\end{proof}

\begin{corollary}\label[corollary]{cor:largeness-partition-verysimple}
Let $X$ be a finite set and $n \in \NN$.
If $X$ is $\bbomega^{n+2}$-large${}^*(\theta)$ and exp-sparse, and $f : X \to \min X$ is a coloring, then there is an $f$-homogeneous $\bbomega^n$-large${}^*(\theta)$ subset.
\end{corollary}
\begin{proof}
If $X$ is $\bbomega^{n+2}$-large${}^*(\theta)$, then $X \setminus \{\min X\}$ is $\bbomega^{n+1} \cdot \min X$-large${}^*(\theta)$, so apply \Cref{cor:largeness-partition-simple}.
\end{proof}

\subsection{Partition theorem for regular trees}

The purpose of this section is to prove a Ramsey-type theorem for colorings of pairs of leaves of a regular tree (\Cref{prop:partition-theorem-coloring-products}), which will be used in the inductive step of the computations of $\EM$-largeness${}^*(\theta)$ and $\RT^2_2$-largeness${}^*(\theta)$. For this, we will use a finite version of Milliken's tree theorem with primitive recursive bounds~\cite{milliken1979ramsey,dodos2016ramsey}.

Milliken's tree theorem is formulated in terms of \emph{strong subforests}, which is a notion of subforest perserving some structural properties of the original forest. More precisely, nodes at the same level of a strong subforest have to come from the same level in the original forest, and the branching degree must be preserved. In this section, we shall exclusively work with strong suforests of $(\bstr, \dots, \bstr)$, and define the notion of strong subforest only in this specialized case.

\begin{definition}
A finite sequence $\vec{T} = (T_0, \dots, T_{d-1})$ of finite subsets of $\bstr$ is a \emph{strong $d$-forest} if there is some~$\ell \in \NN$, a \emph{level function} $h : \{0, \dots, \ell\} \to \NN$ and a finite sequence of functions $(f_0, \dots, f_{d-1})$ with $f_i : 2^{\leq \ell} \to T_i$ such that for every~$i < d$
\begin{itemize}
    \item[(1)] for every~$\sigma \in 2^{\leq \ell}$, $|f_i(\sigma)| = h(|\sigma|)$;
    \item[(2)] for every~$\sigma \in 2^{<\ell}$ and $i < 2$, $f_i(\sigma \cdot i) \succeq f_i(\sigma) \cdot i$.
\end{itemize}
We call~$\ell$ the \emph{height} of~$\vec{T}$, and write $\Leaves(T_i) = \{ f_i(\sigma) : \sigma \in 2^\ell \}$ for the set of \emph{leaves} of~$T_i$. 
\end{definition}

\begin{remark}
Note that contrary to computability theory in which a tree is classically required to be closed under prefix, the notion of tree in Milliken's tree theorem has to be understood as the structure $(S, \preceq)$ for a subset $S \subseteq \bstr$. 
\end{remark}

Given a strong $d$-forest $\vec{T} = (T_0, \dots, T_{d-1})$, we write $\Str_\ell(\vec{T})$ for the set of all strong $d$-forests $\vec{S} = (S_0, \dots, S_{d-1})$ of depth~$\ell$ such that for every~$i < d$, $S_i \subseteq T_i$.
The following theorem is the restriction of Dodos et al.~\cite[Theorem 3.28]{dodos2016ramsey} to strong forests. We write $\E^n$ for the $n$th level of the Grzegorczyk hierarchy~\cite{grzegorczyk1953some} of primitive recursive functions.

\begin{theorem}[Finite Milliken's tree theorem (see~\cite{dodos2016ramsey})]\label[theorem]{thm:finite-milliken}
For every integer~$d \geq 1$ and every triple $\ell, k, r$ of positive integers with $\ell \geq k$, there exists a positive integer~$N$ with the following property. If $\vec{T} = (T_0, \dots, T_{d-1})$ is a strong $d$-forest of depth~$N$, then for every coloring $f : \Str_k(\vec{T}) \to r$, there exists $\vec{S} \in \Str_\ell(\vec{T})$ such that the set $\Str_k(\vec{S})$ is $f$-monochromatic. The least positive integer with this property will be denoted by $\Mil(d, \ell, k, r)$ and is primitive recursive in its parameters, belonging to the class $\mathcal{E}^7$.
\end{theorem}

In our use of Milliken's tree theorem, we shall restrict ourselves to strong $d$-subforests whose leaves have the same length as the parent strong $d$-forest.
Given a strong $d$-forest $\vec{T} = (T_0, \dots, T_{d-1})$, we write $\Str^l_n(\vec{T})$ for the set of all strong $d$-forests $\vec{S} = (S_0, \dots, S_{d-1})$ of depth~$n$ such that for every~$i < d$, $\Leaves(S_i) \subseteq \Leaves(T_i)$.

\begin{corollary}\label[lemma]{cor:milliken-leaves-included}
    Let $d,\ell,k,r \in \NN$ and let $M = \Mil(d,\ell,k,r)$. If $\vec{T} = (T_0, \dots,\allowbreak T_{d-1})$ is a strong $d$-forest of depth~$M$, then for every coloring $f : \Str_k^l(\vec{T}) \to r$ there exists some $\hat{\vec{S}} \in \Str_{\ell}^l(\vec{T})$ such that $\Str_k^l(\hat{\vec{S}})$ is $f$-monochromatic. 
\end{corollary}

\begin{proof}
Fix such family $\vec{T}$ and coloring $f$.

For $\sigma \in T_i$ for some $i < d$, let $\hat{\sigma} \in \Leaves(T_i)$ be the extension of $\sigma$ in $T_i$ by its leftmost path. Similarly, for $R \in \Str(T_i)$, define $\hat{R} \in \Str^l(T_i)$ to be the tree obtained by extending all the leaves of $R$ by their leftmost path in $T_i$. For $\vec{R} = (R_0, \dots, R_{d-1}) \in \Str(\vec{T})$, let $\hat{\vec{R}} := (\hat{R}_0, \dots, \hat{R}_{d-1}) \in \Str^l(\vec{T})$.

Let $\hat{f} : \Str_k(\vec{T}) \to r$ be defined by $\hat{f}(\vec{R}) := f(\hat{\vec{R}})$. By \Cref{thm:finite-milliken}, there exists $\vec{S} \in \Str_{\ell}(\vec{T})$ such that $\Str_k(\vec{S})$ is monochromatic for $\hat{f}$.
By definition of $\hat{f}$, $\hat{\vec{S}}$ is monochromatic for $f$. Indeed, every element of $\Str_k^l(\hat{\vec{S}})$ is of the form $\hat{\vec{R}}$ for some $\vec{R} \in \Str_k(\vec{S})$ and $f(\hat{\vec{R}}) = \hat{f}(\vec{R})$ by definition of $f$.
\end{proof}

\begin{lemma}\label[lemma]{lem:singleton-coloring-one-step}
    Let $d, \ell, r \in \NN$, there exists a primitive recursive bound $B = B_{\ref{lem:singleton-coloring-one-step}}(d,\ell, r)$, such that: if $\vec{T} = (T_0, \dots, T_{d-1})$ is a strong $d$-forest of depth~$B$, then for every coloring $g : \Leaves(T_0) \cup \dots \cup \Leaves(T_{d-1}) \to r$, there exists some $\vec{S} = (S_0, \dots, S_{d-1}) \in \Str_{\ell}^l(\vec{T})$ such that for every $i < d$ and every $\sigma,\sigma' \in \Leaves(S_i)$, $g(\sigma) = g(\sigma')$.
\end{lemma}

\begin{proof}
Let $B_{\ref{lem:singleton-coloring-one-step}}(d,\ell,r) = \Mil(d,\ell,1,r^d)$. Fix such vector $\vec{T}$ and coloring~$f$.
Consider the coloring $h : \Str_1^l(\vec{T}) \to r^d$ defined by:
$$h : (\sigma_0, \dots, \sigma_{d-1}) \mapsto (g(\sigma_0), \dots, g(\sigma_{d-1}))$$
Then, by \Cref{cor:milliken-leaves-included}, there exists some $\vec{S} = (S_0, \dots, S_{d-1}) \in \Str_{\ell}^l(\vec{T})$ such that $\Str_1^l(\vec{S})$ is $h$-monochromatic. Hence, $g$ is constant on every $\Leaves(S_i)$ for $i < d$.
\end{proof}

\begin{lemma}\label[lemma]{lem:pairs-coloring-one-step}
    Let $d,\ell \in \NN$, there exists a primitive recursive bound $B = B_{\ref{lem:pairs-coloring-one-step}}(d, \ell)$, such that: if $\vec{T} = (T_0, \dots, T_{d-1})$ is a strong $d$-forest of depth~$B$, then for every coloring $f : [\Leaves(T_0) \cup \dots \cup \Leaves(T_{d-1})]^2 \to 2$, there exists some $\vec{S} = (S_0, \dots, S_{d-1}) \in \Str_{\ell}^l(\vec{T})$ such that for every $i, j < d$ and every $\sigma,\sigma' \in \Leaves(S_i)$ and $\tau,\tau' \in \Leaves(S_j)$, $f(\sigma,\tau) = f(\sigma',\tau')$.
\end{lemma}

\begin{proof}
Let $B_{\ref{lem:pairs-coloring-one-step}}(d,\ell) = \Mil(d,\Mil(d,\ell,1,2^{\frac{d(d-1)}{2}}),2,2^d)$. Fix such vector $\vec{T}$ and coloring~$f$.
First, notice that for a tree $T$, the sets $\Str_{1}^l(T)$ and $\Str_{2}^l(T)$ are in one-to-one correspondence with $\Leaves(T)$ and $[\Leaves(T)]^2$, respectively. Thus, $\Str_1^l(\vec{T})$ can be seen as the set $\{(\sigma_0, \dots, \sigma_{d-1}) : \sigma_i \in \Leaves(T_i) \}$ and $\Str_2^l(\vec{T})$ can be seen as the set $\{((\sigma_0,\tau_0), \dots, (\sigma_{d-1}, \tau_{d-1})) : \sigma_i,\tau_i \in \Leaves(T_i), \sigma_i < \tau_i\}$.

Consider the two colorings: $h_0 : \Str_2^l(\vec{T}) \to 2^d$ and $h_1 : \Str_1^l(\vec{T}) \to 2^{\frac{d(d-1)}{2}}$, defined by
$$h_0 : ((\sigma_0,\tau_0), \dots, (\sigma_{d-1}, \tau_{d-1})) \mapsto (f(\sigma_0, \tau_0), \dots, f(\sigma_{d-1}, \tau_{d-1}))$$
$$h_1 : (\sigma_0, \dots, \sigma_{d-1}) \mapsto (f(\sigma_0,\sigma_1),f(\sigma_0,\sigma_2), \dots, f(\sigma_{d-2},\sigma_{d-1})) $$
Then, by \Cref{cor:milliken-leaves-included}, there exists some $\vec{S} = (S_0, \dots, S_{d-1}) \in \Str_{\ell}^l(\vec{T})$ such that $\Str_1^l(\vec{S})$ is $h_1$-monochromatic, and such that $\Str_2^l(\vec{S})$ is $h_0$-monochromatic. 
Those two constraints gives us the desired result: the value of $f(x,y)$ for $x \in \Leaves(S_i)$ and $y \in \Leaves(S_j)$ only depends on the pair $(i,j)$. 
\end{proof}

We are now ready to prove our Ramsey-type theorem for regular trees.

\begin{proposition}\label[proposition]{prop:partition-theorem-coloring-products}
    Let $n,x \in \NN$. There exists a primitive recursive bound $B = B_{\ref{prop:partition-theorem-coloring-products}}(n, x)$ such that, for $(T, \preceq) \cong  (B^{\leq 2n-1}, \preceq)$, and for every coloring $f : [\Leaves(T)]^2 \to 2$, there exists a subtree $S \subseteq T$  such that $(S, \preceq) \cong (x^{\leq n}, \preceq)$, $\Leaves(S) \subseteq \Leaves(T)$ and $\Leaves(S)$ is $f$-homogeneous.
\end{proposition}

\begin{proof}
Consider the sequence of integers $M_0, \dots, M_{2n-2}$ defined inductively as follows:
$$\begin{cases}
    M_{i} = B_{\ref{lem:singleton-coloring-one-step}}(2^{M_0 + \dots + M_{i-1}},B_{\ref{lem:pairs-coloring-one-step}}(2^{M_0 + \dots + M_{i-1}},x),2^{2n-1-i}) & \textit{ for } i < 2n-2 \\
    M_{2n-2} = B_{\ref{lem:pairs-coloring-one-step}}(2^{M_0 + \dots + M_{2n-3}},x) & 
\end{cases} $$

Note that, in the case $i = 0$, the sum $M_0 + \dots + M_{i-1}$ is equal to $0$, hence $M_0 = B_{\ref{lem:singleton-coloring-one-step}}(1, B_{\ref{lem:pairs-coloring-one-step}}(1, x), 2^{2n-1})$.
Then, let $B_{\ref{prop:partition-theorem-coloring-products}} = 2^{M_{2n-2}}$ and let $T$ be isomorphic to $B_{\ref{prop:partition-theorem-coloring-products}}^{\leq 2n-1}$. Since $M_i \leq M_{2n-2}$ for every $i \leq 2n-2$, we can then trim $T$ and consider a subtree $T'$ of $T$ such that the nodes of $T'$ at level $i  < 2n-2$ are $2^{M_i}$ branching. We can then view $T'$ as a $2$-branching tree of depth $M_{0} + \dots + M_{2n-2}$ by replacing every $2^{M_i}$-branching by $M_{i}$ consecutive $2$-branching.

Let $P_0 = \{\epsilon\}$, that is, the singleton root of~$T'$ and let $d_0 = 1 = \card P_0$. For $1 \leq k \leq 2n-1$, let $P_k$ be the nodes of $T'$ at level $M_0 + \dots + M_{k-1}$ and let $d_k = 2^{M_0 + \dots + M_{k-1}} = \card P_k$. For $0 \leq k \leq 2n-2$, let $\vec{T}^{k} = (T^{k}_0, \dots, T^{k}_{d_k - 1})$ be the family of subtrees of $T'$ beneath the elements of $P_k$ that corresponds to the levels between $M_0 + \dots + M_{k-1}$ and $M_0 + \dots + M_{k}$ of $T'$. Notice that $P_{k+1} = \Leaves(T_0^k) \cup \dots \cup \Leaves(T_{d_k-1}^k)$ for every $k \leq 2n-2$, and that $P_{2n-1} = \Leaves(T')$.

We can then define two families of colorings $f^k : [P_k]^2 \to 2$ and $g^k : P_k \to 2^{2n-1-k}$, as well as a family of forest $\vec{S}^k = (S_0^k, \dots, S_{d_k - 1}^k) \in \Str_x^l(\vec{T}^k)$, for $k \leq 2n-2$ inductively as follows: 

We have $M_{2n-2} = B_{\ref{lem:pairs-coloring-one-step}}(d_{2n-2},x)$, hence, by \Cref{lem:pairs-coloring-one-step}, there exists some $\vec{S}^{2n-2} = (S_0^{2n-2}, \dots, S^{2n-2}_{d_{2n-2} - 1}) \in \Str_{x}^l(\vec{T}^{2n-2})$ such that the value of $f$ on $[\Leaves(S_0^{2n-2}) \cup \dots \cup \Leaves(S^{2n-2}_{d_{2n-2} - 1})]^2$ only depends on the indexes of the $S^{2n-2}_i$ where the leaves belong.

For $\sigma, \tau \in P_{2n-2}$, let $f^{2n-2}(\sigma,\tau) = f(\sigma',\tau')$ for any $\sigma' \succ \sigma$ and any $\tau' \succ \tau$ in $\Leaves(S^{2n-2}_0) \cup \dots \cup \Leaves(S^{2n-2}_{d_{2n-2} - 1})$. For $\sigma \in P_{2n-2}$, let $g^{2n-2}(\sigma) = f(\sigma',\sigma'')$ for any $\sigma',\sigma'' \succ \sigma$ in $\Leaves(S^{2n-2}_0) \cup \dots \cup \Leaves(S^{2n-2}_{d_{2n-2} - 1})$. Note that both $f^{2n-2}$ and $g^{2n-2}$ are well-defined by choice of $\vec{S}^{2n-2}$.

Suppose that $f^{k+1} : [P_{k+1}]^2 \to 2$ and $g^{k+1} : P_{k+1} \to 2^{2n-1-k}$ are defined for some $1 \leq k < 2n-2$. Since 
$M_{k} = B_{\ref{lem:singleton-coloring-one-step}}(d_k,B_{\ref{lem:pairs-coloring-one-step}}(d_k,x),2^{2n-1-k})$,
we can apply successively \Cref{lem:singleton-coloring-one-step} and \Cref{lem:pairs-coloring-one-step} and obtain some $\vec{S}^{k} = (S_0^{k}, \dots, S^{k}_{d_{k} - 1}) \in \Str_{x}^l(\vec{T}^{k})$ such that $g^{k+1}$ is constant on each $\Leaves(S^k_i)$ and such that the value of $f^{k+1}$ on $[\Leaves(S_0^{k}) \cup \dots \cup \Leaves(S^{k}_{d_{k} - 1})]^2$ only depends on the indexes of the $S^k_i$ where the leaves belong.

For $\sigma, \tau \in P_{k}$, let $f^{k}(\sigma,\tau) = f^{k+1}(\sigma',\tau')$ for any $\sigma' \succ \sigma$ and any $\tau' \succ \tau$ in $\Leaves(S^{k}_0) \cup \dots \cup \Leaves(S^{k}_{d_{k} - 1}) \subseteq P_{k+1}$. For $\sigma \in P_{k}$, let $g^{k}(\sigma) = (f^{k+1}(\sigma',\sigma''), g^{k+1}(\sigma'))$ for any $\sigma',\sigma'' \succ \sigma$ in $\Leaves(S^{k}_0) \cup \dots \cup \Leaves(S^{k}_{d_{k} - 1}) \subseteq P_{k+1}$. Here again, this is well-defined by choice of~$\vec{S}^{k}$.

Finally, consider the subtree $S' \subseteq T'$ equal to $S^0_0$ on levels $0$ to $M_0$, then equal to the $S^1_i$ beneath the leaves of $S^0_0$ on levels $M_0$ to $M_0 + M_1$ and so on. 
\bigskip

\textbf{Claim}: \emph{for every $k \leq 2n-2$, every $\rho \in P_k \cap S'$ with $g^k(\rho) = (c_k, \dots, c_{2n-2})$, and every pair $\sigma, \tau \in \Leaves(S')$ with $\sigma, \tau \succeq \rho$ and with lowest common ancestor at level in the interval $[M_0 + \dots + M_{i-1}, M_0 + \dots + M_i)$, we have $f(\sigma, \tau) = c_i$.} Proceed by induction on $k$, if $k = 2n-2$, then let $\rho \in P_{2n-2} \cap S'$ and write $g^{2n-2}(\rho) = c_{2n-2}$. By construction, $g^{2n-2}(\rho) = f(\sigma,\tau) = c_{2n-2}$ for every $\sigma, \tau \succ \rho$ in $\Leaves(S')$. Let $k < 2n-2$ and assume the property to be true for $k +1$, let $\rho \in P_{k} \cap S'$ and write $g^{k}(\rho) = (c_k, \dots, c_{2n-2})$. Let $\sigma, \tau \in \Leaves(S')$ with $\sigma, \tau \succeq \rho$ and with lowest common ancestor at level in the interval $[M_0 + \dots + M_{i-1}, M_0 + \dots + M_i)$ for some $i \leq 2n-2$, since $\sigma, \tau \succeq \rho$, we have necessarily $i \geq k$. If $i > k$, then let $\rho' \in P_{k+1} \cap S'$ be the common prefix of $\sigma$ and $\rho$ in $P_{k+1}$, by definition of $g^k$, $g^{k+1}(\rho') = (c_{k+1}, \dots, c_{2n-2})$ and by the induction hypothesis, $f(\sigma, \rho) = c_i$. If $i = k$, then let $\sigma_{k+1}, \dots, \sigma_{2n-2}$ and $\tau_{k+1}, \dots, \tau_{2n-2}$ be the prefixes of $\sigma$ and $\tau$ belonging to $P_{k+1}, \dots, P_{2n-2}$ respectively, then $c_k = f^{k+1}(\sigma_{k+1}, \tau_{k+1}) = f^{k+2}(\sigma_{k+2}, \tau_{k+2}) = \dots = f^{2n-2}(\sigma_{2n-2}, \tau_{2n-2}) = f(\sigma, \tau)$. This completes the proof of the claim.
\bigskip

Thus, $P_0$ is a singleton (the root of $T'$), and the value taken by $g^0$ on its element is a $(2n-1)$-tuple $(c_0, \dots, c_{2n-2}) \in 2^{2n-1}$, such that, for $\sigma,\tau$ two leaves of $S'$ with lowest common ancestor at a level in the interval $[M_0 + \dots + M_{i-1}, M_0 + \dots + M_i)$, $f(\sigma, \tau) = c_i$.

$S'$ is a binary tree of depth $x \times (2n-1)$, and by grouping $x$ successive $2$-branching into one, $S'$ can be seen as a $2^x$-branching subtree of $T$ of depth $2n-1$. As $2^x \geq x$, we can trim $S'$ and consider it to be $x$-branching. By the finite pigeonhole principle, there exists some $c < 2$ such that $c_i = c$ for $n$ values of $i < 2n-1$, hence, by only keeping the nodes of $S'$ at those levels (and the final level so that $\Leaves(S) \subseteq \Leaves(T)$), there exists some subtree $S \subseteq S'$, isomorphic to $(x^{\leq n},\preceq)$, such that $f$ is constant on the leaves of $S$.
\end{proof}

\subsection{Grouping principle}

Patey and Yokoyama~\cite{patey2018proof} defined a notion of \qt{grouping} as follows: given a coloring $f : [\NN]^2 \to k$, and two notions of largeness $\L_0, \L_1$, a collection of sets $F_0 < \dots < F_{s-1}$ is an \emph{$(\L_0, \L_1)$-grouping for $f$} if (1) each $F_i$ is $\L_0$-large, (2) $\{ \max F_i : i < s \}$ is $\L_1$-large, and (3) for every $i < j < s$ and $x_0, x_1 \in F_i$ and $y_0, y_1 \in F_j$, $f(x_0, y_0) = f(x_1, y_1)$. In other words, the behavior of~$f$ over different $F$'s depends only on the indices.
This notion was very useful to give an inductive construction of $\EM$-$\bbomega^n$-largeness. Explicit bounds were later computed by Ko{\l}odziejczyk and Yokoyama~\cite{kolo2020some}.

The notion of grouping does not directly translate to largeness${}^*(\theta)$ because of the induced tree structure at every level, but the following lemma can be considered as its counterpart, and plays a similar role in the construction of an $\EM$-$\bbomega^n$-large${}^*(\theta)$ set.

\begin{lemma}\label[lemma]{lem:new-stabilize-bounds}
    Let $X \subseteq_{fin} \NN$ be $\bbomega^{n+2} \cdot (d, \shortbrace{n}, d)$-large${}^*(\theta)$ and exp-sparse for some $d \in \NN$ such that $2^{d^n} \leq \min X$ and let $f : [X]^2 \to 2$ be a coloring. 
    
    Then, there exists an $\bbomega^n \cdot (d, \shortbrace{n}, d)$-large${}^*(\theta)$ subset $Y$ of $X$, such that, for $(Y_{\sigma})_{\sigma \in d^n}$ the decomposition of $Y$ into $\bbomega^n$-large${}^*(\theta)$ blocs, there exists a coloring $g : [d^n]^2 \to 2$ such that $f(x,y) = g(\sigma, \tau)$ for every $x \in Y_{\sigma}$ and $y \in Y_{\tau}$ for some $\sigma, \tau \in d^n$ with $\sigma \neq \tau$.
\end{lemma}

\begin{proof}
Let $(X_{\sigma})_{\sigma \in d^n}$ be a decomposition of $X$ into $\bbomega^{n+2}$-large${}^*(\theta)$ blocs. Consider the sequence $(f_{\sigma})_{\sigma \in d^n}$ of colorings and the sequence $(Y_{\sigma})_{\sigma \in d^n}$ of sets defined inductively as follows:

Let $\sigma \in d^n$ and assume that $Y_{\tau}$ have been defined for every $\tau >_{lex} \sigma$, and consider the two colorings $f_{\sigma}^0 : X_{\sigma} \to 2^{\sum_{\tau <_{lex} \sigma} |X_{\tau}|}$ and $f_{\sigma}^1 : X_{\sigma} \to 2^{|\{\tau \in d^n : \tau >_{lex} \sigma\}|}$ defined respectively by 
$$f_{\sigma}^0(x) = (f(y,x))_{y \in \bigcup_{\tau <_{lex} \sigma} X_{\tau}} \mbox{ and } f_{\sigma}^1(x) = (f(x,y))_{y \in \{\min Y_{\tau} : \tau >_{lex} \sigma\}}$$ 
for every $x \in X_{\sigma}$. Finally, let 
$$f_{\sigma} : X_{\sigma} \to 2^{\sum_{\tau <_{lex} \sigma} |X_{\tau}| + |\{\tau \in d^n : \tau >_{lex} \sigma\}|}$$
be defined by $f_{\sigma}(x) = (f^0_{\sigma}(x),f^1_{\sigma}(x))$.
As $\min X \geq 2^{d^n}$ and as the $X_{\tau}$ are disjoint, we get that $\sum_{\tau <_{lex} \sigma} |X_{\tau}| + |\{\tau : \tau \in d^n \wedge \tau >_{lex} \sigma\} < |[0, \min X) \cup \bigcup_{\tau <_{lex} \sigma} X_{\tau}|$, hence 
$$2^{\sum_{\tau <_{lex} \sigma} |X_{\tau}| + |\{\tau \in d^n : \tau >_{lex} \sigma\}|} \leq \min X_{\sigma}$$
by exp-sparsity of $X$.
Then, by \Cref{cor:largeness-partition-verysimple} there exists some $\bbomega^n$-large${}^*(\theta)$ and $f_{\sigma}$-homogeneous subset $Y_{\sigma} \subseteq X_{\sigma}$.
By definition of $(Y_{\sigma})_{\sigma \in d^n}$, for $\sigma <_{lex} \tau$ and $x,x' \in Y_{\sigma}$ and $y,y' \in Y_{\tau}$, $f(x,y) = f(x, \min Y_{\tau}) = f(x',\min Y_{\tau}) = f(x',y')$.
\end{proof}

\subsection{Erd\H{o}s-Moser theorem}

Thanks to the partition theorem on trees (\Cref{prop:partition-theorem-coloring-products}) and the grouping-like  \Cref{lem:new-stabilize-bounds},
we can now prove a central theorem in this article, namely, a polynomial bound for $\EM$-$\bbomega^{k}$-largeness${}^*(\theta)$. In \Cref{subsect:rt22-new-largeness}, it will be combined with the existing known polynomial bound for $\ADS$-$\bbomega^k$-largeness${}^*(\theta)$ and obtain a bound for $\RT^2_2$-$\bbomega^k$-largeness${}^*(\theta)$.

\begin{theorem}\label[theorem]{thm:em-largeness-star-theta-bound}
There is a primitive recursive function~$g_{\ref{thm:em-largeness-star-theta-bound}} : \NN \to \NN$ such that if $X \subseteq \NN$ is $\bbomega^{6k}$-large${}^*(\theta)$ and $g_{\ref{thm:em-largeness-star-theta-bound}}$-sparse, then it is $\EM$-$\bbomega^{k}$-large${}^*(\theta)$.
\end{theorem}

\begin{proof}
Let $g_{\ref{thm:em-largeness-star-theta-bound}}$ be the map $n \mapsto B_{\ref{prop:partition-theorem-coloring-products}}(n, n)$.
We proceed by induction on $k$. The case $k = 0$ is immediate.

Assume the property to be true for some $k \in \NN$. Let $X$ be an $\bbomega^{6(k+1)}$-large${}^*(\theta)$ set and $f : [X]^2 \to 2$ be a coloring. There exist two $\bbomega^{6k+5}$-large${}^*(\theta)$ subsets $X_0 < X_1$ of $X$. Let $x_0 = \max X_0$. By \Cref{cor:largeness-partition-verysimple}, there exists some $\bbomega^{6k+3}$-large${}^*(\theta)$ subset $Y_1 \subseteq X_1$ such that $f(x_0, y) = f(x_0, y')$ for every $y,y' \in Y_1$.

Note that $x_0 > 6k$ since $x_0 = \max X_0$ which is $\bbomega^{6k+5}$-large${}^*(\theta)$. By $g_{\ref{thm:em-largeness-star-theta-bound}}$-sparsity of~$X$, $B_{\ref{prop:partition-theorem-coloring-products}}(x_0, k) < g_{\ref{thm:em-largeness-star-theta-bound}}(x_0) < \min Y_1$. Thus $Y_1$ is 
$$\bbomega^{6k+2}\cdot(B_{\ref{prop:partition-theorem-coloring-products}}(x_0, k), \shortbrace{6k}, B_{\ref{prop:partition-theorem-coloring-products}}(x_0, k))\text{-large}{}^*(\theta)$$ 
and by \Cref{lem:new-stabilize-bounds}, there exists an $\bbomega^{6k}\cdot(B_{\ref{prop:partition-theorem-coloring-products}}(x_0, k), \shortbrace{6k}, B_{\ref{prop:partition-theorem-coloring-products}}(x_0, k))$-large${}^*(\theta)$ subset $Z_1$ of $Y_1$ such that, for $(Y_{\sigma})_{\sigma \in B_{\ref{prop:partition-theorem-coloring-products}}(x_0, k)^{6k}}$ its decomposition into $\bbomega^{6k}$-large${}^*(\theta)$ blocks, the value of $f(x,y)$ for $x \in Y_{\sigma}$ and $y \in Y_{\tau}$ only depends on the indexes $\sigma, \tau$ when $\sigma \neq \tau$.

By \Cref{lem:product-is-tree}, we can view $Z_1$ as a $B_{\ref{prop:partition-theorem-coloring-products}}(x_0, k)$-branching tree $T$ of depth $6k$ with leaves labelled with the $Y_{\sigma}$. The coloring $f$ induces a coloring~$\hat f$ on the pairs of leaves of that tree, and, by \Cref{prop:partition-theorem-coloring-products}, as $6k \geq 2k$ there exists a subtree $S \subseteq T$ such that $(S, \preceq) \cong (x_0^{\leq k}, \preceq)$, $\Leaves(S) \subseteq \Leaves(T)$, and such that $\Leaves(S)$ is homogeneous for~$\hat f$. 

We can then apply the inductive hypothesis on each $Y_{\sigma}$, for $\sigma \in \Leaves(S)$ and obtains $\bbomega^{k}$-large${}^*(\theta)$ subsets $Z_{\sigma} \subseteq Y_{\sigma}$ that are $f$-transitive. Let $H = \{x_0\} \cup \bigcup_{\sigma \in \Leaves(S)} Z_{\sigma}$, by construction $H$ is $\bbomega^{k+1}$-large${}^*(\theta)$ as every $Z_{\sigma}$ is $\bbomega^{k}$-large${}^*(\theta)$ and as $S$ is a subtree of $T$.

We claim that the set $Z$ is $f$-transitive. Indeed, pick by contradiction some $x < y < z \in Z$ such that $f(x,y) = f(y,z)$ and $f(x,y) \neq f(x,z)$. We cannot have $x = x_0$, as $f(x_0,y) = f(x_0,z)$ for every $y,z \in Y_1$. We cannot have $x,y,x \in Z_{\sigma}$ for some $\sigma \in \Leaves(S)$ by our application of the inductive hypothesis. We cannot have $x,y \in Z_{\sigma_0}$ and $z \in Z_{\sigma_1}$ for some $\sigma_0 <_{lex} \sigma_1 \in \Leaves(S)$ by our application of \Cref{lem:new-stabilize-bounds}, and similarly we cannot have $x \in Z_{\sigma_0}$ and $y,z \in Z_{\sigma_1}$ for some $\sigma_0 < \sigma_1 \in \Leaves(S)$. Finally, we cannot have $x \in Z_{\sigma_0}$, $y \in Z_{\sigma_1}$ and $z \in Z_{\sigma_2}$ for some $\sigma_0 <_{lex} \sigma_1 <_{lex} \sigma_2 \in \Leaves(S)$, as by our application \Cref{lem:new-stabilize-bounds} we have $f(x,y) = \hat f(\sigma_0, \sigma_1)$, $f(y,z) = \hat f(\sigma_1, \sigma_2)$ and $f(x,z) = \hat f(\sigma_0, \sigma_2)$ and by \Cref{prop:partition-theorem-coloring-products}, $\hat f$ is homogeneous on $\Leaves(S)$. Hence, every case is impossible and $Z$ is $f$-transitive. This completes our claim and the proof of \Cref{thm:em-largeness-star-theta-bound}.
\end{proof}

\subsection{Ramsey's theorem for pairs}\label[section]{subsect:rt22-new-largeness}

Le Houérou, Levy Patey and Yokoyama~\cite[Proposition 5.4]{houerou2026pi} proved that every $\bbomega^{4n+4}$-large$(\theta)$ set is $\ADS$-$\bbomega^n$-large$(\theta)$, using the theorem of Ketonen and Solovay~\cite{ketonen1981rapidly}. By replacing its use with \Cref{thm:better-ketonen-solovay-bound}, this yields the following bound for $\ADS$:

\begin{proposition}[Le Houérou, Levy Patey and Yokoyama~\cite{houerou2026pi} revisited, $\ISig_1$]\label[proposition]{prop:ads-old-largeness}
Let $n \geq 1$. If $F \finsub \NN$ is $\bbomega^{4n} \cdot 2$-large$(\theta)$ and $(x \mapsto 2x^{x^2})$-sparse, then it is $\ADS$-$\bbomega^n$-large$(\theta)$.
\end{proposition}

Thanks to the translation between largeness$(\theta)$ and largeness${}^*(\theta)$, one obtains a polynomial bound for $\ADS$
in terms of largeness${}^*(\theta)$:

\begin{corollary}\label[corollary]{cor:ads-new-largeness}
Let $n \geq 1$. If $F \finsub \NN$ is $\bbomega^{2n^2+2n} \cdot 2$-large${}^*(\theta)$ and $(x \mapsto 2x^{x^2})$-sparse, then it is $\ADS$-$\bbomega^n$-large${}^*(\theta)$.
\end{corollary}
\begin{proof}
Since $F$ is $\bbomega^{4\frac{n(n+1)}{2}} \cdot 2$-large${}^*(\theta)$, it is also $\bbomega^{4\frac{n(n+1)}{2}} \cdot 2$-large$(\theta)$,
so by \Cref{prop:ads-old-largeness}, it is $\ADS$-$\bbomega^{\frac{n(n+1)}{2}}$-large$(\theta)$. By \Cref{prop:large-to-new-large}, $F$ is $\ADS$-$\bbomega^n$-large${}^*(\theta)$.
\end{proof}

We are now ready to give a polynomial bound to $\RT^2_2$-$\bbomega^n$-largeness${}^*(\theta)$.

\begin{theorem}\label[theorem]{thm:rt22-largeness-star-theta-bound}
There is a primitive recursive function $g_{\ref{thm:rt22-largeness-star-theta-bound}} : \NN \to \NN$ such that
for every $n \in \NN$ and for every finite set $X \subseteq \NN$, if $X$ is $\bbomega^{12n^2 + 12n + 6}$-large${}^*(\theta)$ and $g_{\ref{thm:rt22-largeness-star-theta-bound}}$-sparse, then $X$ is $\RT_2^2$-$\bbomega^n$-large${}^*(\theta)$.
\end{theorem}

\begin{proof}
Let $g_{\ref{thm:rt22-largeness-star-theta-bound}} : \NN \to \NN$ be defined by $g_{\ref{thm:rt22-largeness-star-theta-bound}}(x) = \max(2x^{x^2}, g_{\ref{thm:em-largeness-star-theta-bound}}(x))$.
Fix a coloring $f : [X]^2 \to 2$. $X$ being $\bbomega^{6(2n^2 + 2n + 1)}$-large${}^*(\theta)$ and $g_{\ref{thm:rt22-largeness-star-theta-bound}}$-sparse, by \Cref{thm:em-largeness-star-theta-bound} there exists some $f$-transitive and $\bbomega^{2n^2 + 2n + 1}$-large${}^*(\theta)$ subset $Y \subseteq X$.
By \Cref{cor:ads-new-largeness}, there exists some $f$-homogeneous and $\bbomega^n$-large$^*(\theta)$ subset $Z \subseteq Y$. 
Thus, $X$ is $\RT_2^2$-$\bbomega^n$-large${}^*(\theta)$.
\end{proof}

\begin{corollary}\label[corollary]{cor:rt22-largeness-star-theta-bound-no-sparsity}
    There exists a polynomial $P \in \mathbb{Z}[X]$, such that, for every $n \in \NN$, if $X \subseteq \NN$ is an $\bbomega^{P(n)}$-large${}^*(\theta)$ finite set, then $X$ is $\RT_2^2$-$\bbomega^n$-large${}^*(\theta)$.
\end{corollary}

\begin{proof}
By \Cref{cor:new-large-g-sparse}, let $k_0 \in \NN$ be such that, for every $n \in \NN$, every $\bbomega^{n+k_0}$-large${}^*(\theta)$ set contains a $g_{\ref{thm:rt22-largeness-star-theta-bound}}$-sparse $\bbomega^n$-large${}^*(\theta)$ subset. 

Let $P(X) = 12X^2 + 12X + 6 + k_0$ and fix some $n \in \NN$. Let $X \subseteq \NN$ be $\bbomega^{P(n)}$-large${}^*(\theta)$, by definition of $k_0$, there exists some $\bbomega^{12n^2 + 12n + 6}$-large${}^*(\theta)$ and $g_{\ref{thm:rt22-largeness-star-theta-bound}}$-sparse subset $Y$ of $X$. By \Cref{thm:rt22-largeness-star-theta-bound}, $Y$ is $\RT_2^2$-$\bbomega^n$-large${}^*(\theta)$, thus $X$ is also $\RT_2^2$-$\bbomega^n$-large${}^*(\theta)$.
\end{proof}

Note that instead of translating the linear bound for $\ADS$-$\bbomega^n$-largeness$(\theta)$ (\Cref{prop:ads-old-largeness}) into a polynomial bound for $\ADS$-largeness${}^*(\theta)$ (\Cref{cor:ads-new-largeness}), and combine it with the linear bound for $\EM$-$\bbomega^n$-largeness${}^*(\theta)$ (\Cref{thm:em-largeness-star-theta-bound}) to obtain a polynomial bound for $\RT^2_2$-largeness${}^*(\theta)$ (\Cref{thm:rt22-largeness-star-theta-bound}), one could rather have translated the bound for $\EM$ into the other notion of largeness, and obtained a polynomial bound for $\RT^2_2$-largeness$(\theta)$. This shows that the previous notion of largeness$(\theta)$ still admits enjoyable properties in terms of closure under combinatorial theorems, but the largeness$(\theta)$ proofs do not seem to follow naturally the combinatorial proofs of the statements.

\section{Polynomial simulation of $\RCA_0 + \RT^2_2$}\label[section]{sec:translation}

In this section, we use the general framework of forcing interpretation as formulated in \cite[Section 1]{kolodziejczyk2023ramsey} to prove \Cref{thm:rt22-polynomially-simulated}. The general idea of a forcing interpretation extends that of an interpretation of a theory $T'$ from another theory $T$. Indeed, while an interpretation formalizes the construction of a model $T'$ inside a model of~$T$, a forcing interpretation formalizes the definition of a generic model of $T'$ from a model of~$T$ through forcing. 

As mentioned in \cite{kolodziejczyk2023ramsey}, contrary to most forcing translations which build a model $\bbomega$-extending the ground model, the main forcing interpretation of this series of translations formalizes the construction of a generic proper cut~$I \prec_e M$ and considers the model $(M, \Cod(I/M))$.

Due to the heavy formalism of forcing interpretation, we shall sacrifice self-containment for the sake of concision. We therefore expect the reader to be familiar with the constructions of \cite[Section 2]{kolodziejczyk2023ramsey}, and we are mostly going to emphasize the differences in the adaptation.

Our first change will concern the theory $\mathbb{I}$ given by \cite[Definition 2.7]{kolodziejczyk2023ramsey}, which we replace by a theory $\mathbb{I}'$, requiring the corresponding cut to be also closed under multiplication (the bound obtained in \Cref{cor:rt22-largeness-star-theta-bound-no-sparsity} being given by a polynomial in $n$ of degree $2$ as opposed to the bound $144n + O(1)$ obtained \cite{kolo2020some}), and taking into account the fact that our notion of largeness now depends on a parameter~$\theta$.

\begin{definition}
Let $\mathbb{I}'$ be the theory consisting of the following axioms:
\begin{itemize}
    \item[($\mathbb{I}1$)]: $\mathbb{I}$ is a nonempty proper cut in the first-order universe,
    \item[$(\mathbb{I}'2)$]: $\mathbb{I}$ is closed under addition and multiplication, 
    \item[$(\mathbb{I}'3)$]: for every $\Delta_0^0$ formula $\theta(y,z,t)$ and every $a \in \NN$, if $\forall y \exists z \forall t \theta(y,z,t)$ holds, then there exists some $\bbomega^x$-large${}^*(\theta)$ set $s$ with $\min s > a$ for some $x > \mathbb{I}$ .
\end{itemize}
    
    
\end{definition}

\Cref{fig:forcing-translations} summarizes the different intermediate theories we will consider between $\RCA_0 + \BSig_2$ and $\RT_2^2 + \WKL_0$, and the polynomial simulation between them. It is adapted from \cite[Figure 1]{kolodziejczyk2023ramsey}, with $\BSig_2$ added to every theory, and with a $\forall \Sigma_3^0$ polynomial simulation between 
$\RCA_0 + \BSig_2 + \mathbb{I}'$ and $\WKL_0 + \RT_2^2$ instead of a $\forall \Sigma_2^0$ polynomial simulations.

\begin{figure}[tph]

\centering
\begin{tikzpicture}[xscale=2,yscale=1.2]
\node(IS1)    at ( 0,0)   {$\mathsf{B}\Sigma_2$};
\node(RCA0)   at ( 0,1)   {$\RCA_0 + \BSig_2$};
\node[minimum width=8em](RCA0+)
              at (-1,2.5) {$\RCA_0+\ISig_2$};
\node[minimum width=8em](RCA0-)
              at ( 1,2.5) {$\RCA_0+\BSig_2 + \neg\ISig_2$};
\node(RCA0I1) at ( 0,3.5) {$\RCA_0+\BSig_2 + (\mathbb{I}1)$};
\node(RCA0I)  at ( 0,4.5) {$\RCA_0+\BSig_2 + \mathbb{I}'$};
\node(RT22)   at ( 0,5.5) {$\WKL_0+\RT^2_2$};
\begin{scope}[->,node font=\scriptsize]
\draw(IS1)   --node[right]{$\mathcal{L}_1$}                  (RCA0);
\draw(RCA0+) --node[above left]{$\forall\Sigma^0_3$} (RCA0I1);
\draw(RCA0-) --node[above right]{$\mathcal{L}_2$}            (RCA0I1);
\draw(RCA0I1)--node[right]{$\mathcal{L}_2$}                  (RCA0I);
\draw(RCA0I) --node[right]{$\forall\Sigma^0_3$}      (RT22);
\end{scope}
\draw(RCA0)--++(0,.5) coordinate(pivot)
     (pivot)--(RCA0+)
     (pivot)--(RCA0-)
     (pivot)+(45:.2) arc [radius=.2, start angle=45, end angle=135];
\end{tikzpicture}

\caption{Polynomial simulations between the various theories in \Cref{sec:translation}.}
\label[figure]{fig:forcing-translations} 
\end{figure}
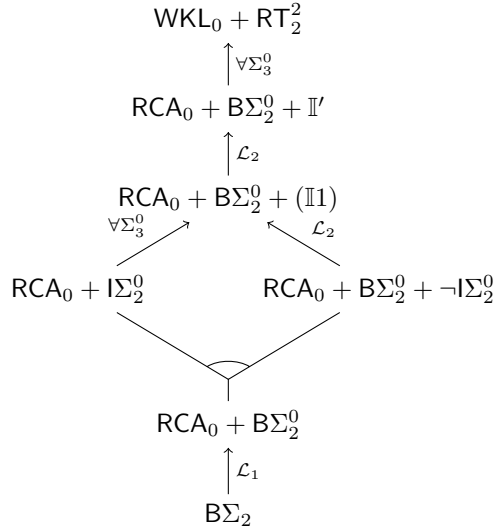

All of these polynomial simulations except the one between $\RCA_0+\BSig_2 + \mathbb{I}'$ and $\WKL_0+\RT^2_2$ are obtained by a traditional interpretation, which can be seen as a degenerate forcing interpretation. In most cases, the interpretation will be a straightforward adaptation of the one in \cite{kolodziejczyk2023ramsey}.

The following result corresponds to \cite[Lemma 2.8]{kolodziejczyk2023ramsey}. Solovay's technique of shortening cuts used here to obtain a cut stable by addition and multiplication (see e.g. \cite[Theorem III.3.5]{hajek1998metamathematics}) cannot be used to construct cuts stable by exponentiation, which is the reason why the exponential bound of \cite{houerou2026pi} was not sufficient to obtain a polynomial simulation. 

\begin{lemma}\label[lemma]{lem:poly-sim-I-I1}
    $\RCA_0 + \BSig_2 + (\mathbb{I}1)$ polynomially simulates $\RCA_0 + \BSig_2 + \mathbb{I}'$ with respect to $\L_2$ sentences.
\end{lemma}

\begin{proof}
We show that there exists an interpretation of $\RCA_0 + \BSig_2 + \mathbb{I}'$ in $\RCA_0 + \BSig_2 + (\mathbb{I}1)$ that is the identity interpretation with respect to all symbols of $\L_2$.

Let $\mathbb{J}$ be the intersection of all the $\Jcut^*_\theta$ for every $\Delta_0^0$ formula $\theta(y,z,t)$ such that $\forall y \exists z \forall t \theta(y,z,t)$ holds.

By \Cref{prop:new-cut-is-additive}, $\RCA_0 + \BSig_2$ proves that $\mathbb{J}$ forms a cut. Notice that for $\theta(y,z,t)$ a $\Delta_0^0$ formula and $a \in \NN$, the formula \qt{there exists an $\bbomega^x$-large${}^*(\theta)$ set $s$ with $\min s > a$} is $\Sigma_1^0$. Since $\ISig_1$ holds, the proper cut $\mathbb{I} \cap \mathbb{J}$ is not $\Sigma^0_1$-definable, hence, if $\forall y \exists z \forall t \theta(y,z,t)$ holds, then, for every $a \in \NN$ there must exists some $\bbomega^x$-large${}^*(\theta)$ set $s > a$ for some $x > \mathbb{I} \cap \mathbb{J}$. Thus, $\mathbb{I} \cap \mathbb{J}$ is, provably in $\RCA_0 + \BSig_2 + (\mathbb{I}1)$, a proper definable cut satisfying $(\mathbb{I}1)$ and $(\mathbb{I}'3)$. 

By applying Solovay's technique of shortening cuts, we consider the cuts $\mathbb{K} = \{a \in \mathbb{I} \cap \mathbb{J} : \forall x \in \mathbb{I} \cap \mathbb{J}\ (a + x \in \mathbb{I} \cap \mathbb{J})\}$ and $\mathbb{K}' = \{a \in \mathbb{K} : \forall x \in \mathbb{K}\ (a\cdot x \in \mathbb{K})\}$ that are respectively closed under addition and multiplication. Thus, $\RCA_0 + (\mathbb{I}1)$ proves the axioms of $\RCA_0 + \mathbb{I}'$ with $\mathbb{K}'$ substituted for $\mathbb{I}$.
\end{proof}

The following result corresponds to \cite[Lemma 2.9]{kolodziejczyk2023ramsey}.

\begin{lemma}\label[lemma]{lem:poly-sim-I1}
    $\RCA_0 + \BSig_2 + (\mathbb{I}1)$ is polynomially simulated by:
    \begin{enumerate}[(a)]
        \item $\RCA_0 + \ISig_2$ with respect to $\forall \Sigma_3^0$ sentences,
        \item $\RCA_0 + \BSig_2 + \neg \ISig_2$ with respect to $\mathcal{L}_2$ sentences.
    \end{enumerate}
\end{lemma}

\begin{proof}
The proof of \cite[Lemma 2.9.(a)]{kolodziejczyk2023ramsey} already gives the stronger result with $\RCA_0 + \BSig_2 + (\mathbb{I}1)$ instead of $\RCA_0 + (\mathbb{I}1)$. Indeed, by \cite[Theorem 5.1]{beklemishev1998proof}, $\mathsf{I}\Sigma_2$ proves the uniform $\Pi_4$ reflection principle for $\mathsf{B}\Sigma_2$ and not just $\mathsf{I}\Sigma_1$.

The proof of $(b)$ is the same as that of \cite[Lemma 2.9.(b)]{kolodziejczyk2023ramsey}.
\end{proof}

We now turn to the main polynomial simulation, that is, the proof that $\RCA_0+\BSig_2+\mathbb{I}'$ polynomially simulates $\WKL_0 + \BSig_2$. As explained, this involves the definition of a forcing interpretation, which formalizes the construction of $(I, \Cod(M/I)$ for a generic proper cut~$I$. We start with the definition of the forcing interpretation, which consists of defining a set of conditions $\mathsf{Cond}$, partially ordered by a relation $\trianglelefteqslant$, together with two predicates, namely~$s \Vdash v \downarrow$ for $s \in \mathsf{Cond}$ and $v \in \NN$, and $s \Vdash \alpha(v_0, \dots, v_{\ell-1})$ for every $s \in \mathsf{Cond}$ and every simple atomic formula $\alpha(v_0, \dots, v_{\ell-1})$. 

As our notion of largeness depends on a parameter $\theta$, so will the forcing conditions. A condition will therefore be a pair $(s, \theta)$, such that if $(s, \theta') \trianglelefteqslant (t, \theta)$, then $\theta = \theta'$. It follows that in any filter~$\F$, the second component of each condition shares the same parameter~$\theta$. Thus, the resulting forcing notion will satisfy the same properties as the one defined in \cite[Definition 2.11]{kolodziejczyk2023ramsey}.

\begin{definition}\label[definition]{def:forcing-interpretation-def}\ 
\begin{itemize}
    \item[(i)] A \emph{forcing condition} is a couple $(s,\theta)$ such that $\theta$ is (the code of) a $\Delta_0^0$ formula and $s$ is a finite $\bbomega^x$-large${}^*(\theta)$ set for some $x > \mathbb{I}$. We let $\mathsf{Cond}$ be the set of all forcing conditions.
    \item[(ii)] Given two forcing conditions $(s_0, \theta_0)$ and $(s_1, \theta_1)$, let $(s_0, \theta_0) \trianglelefteqslant (s_1, \theta_1)$ if $\theta_0 = \theta_1$ and $s_0 \subseteq s_1$.
    \item[(iii)] Let $(s,\theta) \Vdash v \downarrow$ if $(s \cap [1,v], \theta)$ is not a condition. And $(s,\theta) \Vdash V \downarrow$ always holds. Given a simple atomic formula $\phi(\bar{v})$, let $(s,\theta) \Vdash \phi$ if $(s \Vdash \bar{v} \downarrow) \wedge \phi(\bar{v})$, where $\bar v$ contains all the parameters of~$\phi$.
\end{itemize}
\end{definition}

The relations $\mathsf{Cond}, \trianglelefteqslant$, $\Vdash$ shares the same properties as their equivalent from \cite[Definition 2.11]{kolodziejczyk2023ramsey} and direct counterparts of \cite[Lemma 2.12, 2.13 and 2.14]{kolodziejczyk2023ramsey} can be proven in the same manner.
The following counterpart of \cite[Lemma 2.15]{kolodziejczyk2023ramsey} exploits in an essential way the polynomial bound obtained in \Cref{cor:rt22-largeness-star-theta-bound-no-sparsity}.

\begin{lemma}\label[lemma]{lem:polynomial-forcing-interpretation-wkl-rt22}
    The relations $\mathsf{Cond}, \trianglelefteqslant$, $\Vdash$ of \Cref{def:forcing-interpretation-def} determine a polynomial forcing interpretation of $\WKL_0 + \RT_2^2$ in $\RCA_0 + \BSig_2 + \mathbb{I}'$.
\end{lemma}

\begin{proof}
Since by the counterpart of \cite[Lemma 2.12]{kolodziejczyk2023ramsey}, $\mathsf{Cond}$, $\trianglelefteqslant$, $\Vdash$ determine a forcing interpretation of pure logic from $\mathcal{L}_2$ in $\RCA_0 + \BSig_2 + \mathbb{I}'$, it remains to show that any condition $(s,\theta)$ forces each of the axioms of $\WKL_0 + \RT_2^2$. The proof is similar to that of \cite[Lemma 2.15]{kolodziejczyk2023ramsey}.

For $\RT_2^2$, we use \Cref{cor:rt22-largeness-star-theta-bound-no-sparsity} instead of \cite[Theorem 2.6]{kolodziejczyk2023ramsey}.

For the other axioms, the only combinatorial property of largeness the proof exploits is \cite[Proposition 2.4(ii)]{kolodziejczyk2023ramsey} in the case of $\omega^n \cdot k$-largeness, which also holds for $\omega^n \cdot k$-largeness${}^*(\theta)$: if $X$ is $\omega^n \cdot k$-large${}^*(\theta)$, there are $k$ many $\omega^n$-large${}^*(\theta)$ subsets $X_0 < \dots < X_{k-1} \subseteq X \setminus \{\min X\}$.
\end{proof}


The main difference between our proof and that of \cite{kolodziejczyk2023ramsey} lies in the following lemma. It corresponds to \cite[Lemma 2.16]{kolodziejczyk2023ramsey} but is formulated for $\forall \Sigma_3^0$-reflection rather than $\forall \Sigma_2^0$-reflection. In their argument, in order to obtain a condition $s$ forcing a given $\Pi_2^0$ formula $\forall y \exists z \theta(y,z)$ holding in some model of $\RCA_0 + \mathbb{I}$, Ko{\l}odziejczyk et al.~ consider an infinite set $\{k_0, k_1, \dots \}$ such that $\forall y < k_i \exists z < k_{i+1} \theta(y,z)$ for every $i$ and then take $s$ to be an $\omega^x$-large subset of that set for some $x > \mathbb{I}$. However, for a $\Pi_3^0$ formula $\forall y \exists z \forall t \theta(y,z,t)$ holding in a model of $\RCA_0 + \BSig_2$, one cannot prove the existence of an infinite set $\{k_0, k_1, \dots \}$ satisfying $\forall y < k_i \exists z < k_{i+1} \forall t \theta(y,z,t)$. Instead, we will choose some $\omega^x$-large${}^*(\theta)$ set $s$ for some $x > \mathbb{I}$, and prove that $(s,\theta)$ forces $\forall y \exists z \forall t \theta(y,z,t)$.

\begin{lemma}\label[lemma]{lem:polynomially-reflecting}
    The forcing interpretation of $\WKL_0 + \RT_2^2$ in $\RCA_0 + \BSig_2 + \mathbb{I}'$ given by $\mathsf{Cond}, \trianglelefteqslant$, $\Vdash$ of \Cref{def:forcing-interpretation-def} is polynomially $\forall \Sigma_3^0$-reflecting.
\end{lemma}

\begin{proof}
We proceed by contrapositive. Let $\phi := \exists X \forall y \exists z \forall t \theta(X,y,z,t)$ be a $\exists \Pi_3^0$ sentence and assume that $\RCA_0 + \BSig_2 + \mathbb{I}' \vdash \phi$. Let $b \in \NN$ be bigger than all the first-order constants appearing in $\phi$. By $(\mathbb{I}'3)$, there exists a set $A$ and some $\bbomega^x$-large${}^*(\theta(A))$ set $s$ with $s > b$ for some $x > \mathbb{I}$. We claim that $(s,\theta(A)) \Vdash \forall y \exists z \forall t \theta(A,y,z,t)$, which would imply that $(s,\theta(A)) \not \Vdash \forall X \exists y \forall z \exists t \neg \theta(X,y,z,t)$ yielding the $\forall \Sigma_3^0$-reflection.


Take $(s',\theta(A)) \trianglelefteqslant (s,\theta(A))$ and $\ell$ such that $(s',\theta(A)) \Vdash \ell \downarrow$. Since $(s'\cap [1,\ell],\theta(A))$ is not a condition, $s' \setminus [1, \ell]$ must be $\omega^{x'}$-large${}^*(\theta)$ for some $x' > \mathbb{I}$. Let $s_0 < s_1$ be two $\bbomega^{x' - 1}$-large${}^*(\theta(A))$ and $\theta(A)$-apart subsets of $s' \setminus [1, \ell]$. We claim that $(s_1, \theta(A)) \Vdash \exists z \forall t \theta(A,\ell,z,t)$.

As $s_0$ and $s_1$ are $\theta(A)$-apart, there exists some $\ell' < \min s_1$ such that $\forall t < \max s_1, \theta(A,\ell,\ell',t)$ holds. As $\ell' < \min s_1$ we have that $s_1 \Vdash \ell' \downarrow$ since $(s_1 \cap [1, \ell'],\theta(A)) = (\emptyset,\theta(A))$ is not a condition. Hence it is sufficient to show that $(s_1, \theta(A)) \Vdash \forall t \theta(A,\ell, \ell',t)$.

Let $\ell''$ be such that $(s_1, \theta) \Vdash \ell'' \downarrow$. Necessarily we have $\ell'' < \max s_1$ otherwise $s_1 \cap [1,\ell''] = \emptyset$ and cannot be $\bbomega^{x''}$-large${}^*(\theta)$ for some $x'' > \mathbb{I}$. Thus, as $s_0$ and $s_1$ are $\theta(A)$-apart, $\theta(A,\ell, \ell', \ell'')$ holds. From the assumption that $s > b$, we have that $(s_1', \theta(A)) \Vdash v \downarrow$ for every first-order parameter $v$ in the formula $\theta(A, \ell, \ell', \ell'')$, hence $(s_1',\theta(A)) \Vdash \theta(A,\ell,\ell',\ell'')$ and $(s_1, \theta(A)) \Vdash \forall t \theta(A,\ell, \ell',t)$. Therefore, $(s_1, \theta(A)) \Vdash \exists z \forall t \theta(A,\ell,z,t)$ and $(s,\theta(A)) \Vdash \forall y \exists z \forall t \theta(A,y,z,t)$.
\end{proof}

We can now combine the previous lemmas and obtain the desired polynomial simulation.

\begin{lemma}\label[lemma]{lem:polynomial-simulation}
    $\RCA_0 + \mathbb{I}'$ polynomially simulates $\WKL_0 + \RT_2^2$ with respects to $\forall \Sigma_3^0$-sentences.
\end{lemma}

\begin{proof}
    Follows directly from \Cref{lem:polynomial-forcing-interpretation-wkl-rt22}, \Cref{lem:polynomially-reflecting} and \cite[Theorem 1.17]{kolodziejczyk2023ramsey}.
\end{proof}

Putting altogether the successive polynomial simulations, we obtain our main theorem:

\begin{repmaintheorem}{thm:rt22-polynomially-simulated}
$\RCA_0 + \BSig_2$ polynomially simulates $\WKL_0 + \RT_2^2$ with respect to $\forall \Sigma_3^0$ sentences.
\end{repmaintheorem}

\begin{proof}
    By \Cref{lem:polynomial-simulation} and \Cref{lem:poly-sim-I-I1} there is a polynomial-time procedure which, given a proof $\pi$ of a $\forall \Sigma_3^0$ sentence $\psi$ in $\WKL_0 + \RT_2^2$, outputs a proof $\pi'$ of $\psi$ in $\RCA_0 + \BSig_2 + (\mathbb{I}1)$. By \Cref{lem:poly-sim-I1}, a further polynomial-time procedure outputs a proof $\pi''$ of $\psi$ in $\RCA_0 + \ISig_2$ and a proof $\pi'''$ of $\psi$ in $\RCA_0 + \BSig_2 + \neg \ISig_2$. Combine $\pi'', \pi'''$, and a case distinction to obtain a proof of $\psi$ in $\RCA_0 + \BSig_2$.
\end{proof}

\bigskip
\begin{center}
\textbf{Acknowledgements}
\end{center}
The authors are thankful to Leszek Ko{\l}odziejczyk and Keita Yokoyama for insightful comments and discussions.

\bibliographystyle{plain}
\bibliography{biblio}

\end{document}